\documentclass[11pt,twoside]{article}
\title{Existence and nonuniqueness of segregated solutions to a class of cross-diffusion systems
\thanks{Supported by the Spanish MCINN Project
MTM2010-18427}}

\author{Gonzalo Galiano  \thanks{Dpt. of Mathematics, Universidad de Oviedo,
 c/ Calvo Sotelo, 33007-Oviedo, Spain ({\tt galiano@uniovi.es, shmarev@orion.ciencias.uniovi.es, julian@uniovi.es})}
    \and Sergey Shmarev \footnotemark[2] \and Juli\'an Velasco \footnotemark[2] }
\date{}
\usepackage{fancyhdr}
\usepackage{amsmath}
  \usepackage{amsthm}
   \usepackage{amssymb}
    \usepackage{paralist}
     \usepackage{graphics}
      \usepackage{epsfig}
       \usepackage[colorlinks=true]{hyperref}
\hypersetup{urlcolor=blue, citecolor=red}
 \usepackage{hyperref}
  \usepackage{color}
   \usepackage[spanish,english]{babel}
    \usepackage[latin1]{inputenc}
     \usepackage{subfigure}

\newtheorem{theorem}{Theorem}[section]
\newtheorem{corollary}{Corollary}

\newtheorem{lemma}[theorem]{Lemma}

\theoremstyle{definition}
\newtheorem{definition}[theorem]{Definition}
\newtheorem{remark}{Remark}

\begin{document}
\maketitle

\begin{abstract}
We study the the Dirichlet problem for the cross-diffusion system
\[
\partial_tu_i=\operatorname{div}\left(a_iu_i\nabla (u_1+u_2)\right)+f_i(u_1,u_2),\quad
i=1,2,\quad a_i=const>0,
\]
in the cylinder $Q=\Omega\times (0,T]$. The functions $f_i$ are assumed to satisfy the
conditions $f_1(0,r)=0$, $f_2(s,0)=0$, $f_1(0,r)$, $f_2(s,0)$ are locally
Lipschitz-continuous. It is proved that for suitable initial data $u_0$, $v_0$ the system
admits segregated solutions $(u_1,u_2)$ such that $u_i\in L^{\infty}(Q)$, $u_1+u_2\in
C^{0}(\overline{Q})$, $u_1+u_2>0$ and $u_1\cdot u_2=0$ everywhere in $Q$. We show that the
segregated solution is not unique and derive the equation of motion of the surface $\Gamma$
which separates the parts of $Q$ where $u_1>0$, or $u_2>0$. The equation of motion of
$\Gamma$ is a modification of the Darcy law in filtration theory. Results of numerical
simulation are presented.
\end{abstract}

\begin{itemize}
 \item[] {\small \emph{Keywords: }
 Nonlinear parabolic equation, cross-diffusion system, segregated solutions, Lagrangian coordinates. }
 
 \item[] {\small \emph{AMS: }35K55, 35K57, 35K65, 35R35 }
\end{itemize}

\section{Introduction}
In the context of Population Dynamics, Gurney and Nisbet \cite{gurney75} derived from
microscopic considerations the density-dependent population flux

\begin{equation}
\notag
 J(u)=c\nabla u + a u \nabla u,
\end{equation}
with positive constants $a$ and $c$. In this expression the term  $c\nabla u$ reflects a
random dispersal of the population, while \emph{the population pressure} $au\nabla u$
prevents overcrowding. The corresponding evolution equation has the form

\begin{equation}
\label{S1}
\partial_t u -\operatorname{div} J(u) =u\left(\alpha -\frac{ u}{\beta}\right),
\end{equation}
where the right-hand side is the logistic growth term, $\alpha>0$ is the intrinsic growth
rate and $\beta>0$ is the carrying capacity.

Various generalizations of this model were proposed, from different points of view, by
Shigesada et al. \cite{shigesada79}, Busenberg and Travis \cite{busenberg83}, or
\cite{gurtin84,galiano12}, among others, and have given rise to the so-called
\emph{cross-diffusion models}. The authors of \cite{busenberg83} assume that the individual
population flow $J_i$ is proportional to the gradient of a potential function $\Psi$ which
depends only on the total population density $U=u_1+u_2$:

\[
J_i(u_1,u_2)=a\frac{u_i}{U}\nabla \Psi (U).
\]
In this model the collective flow is still given in the form \eqref{S1}:
$J(U)=a\nabla\Psi(U)$ with $c=0$. Assuming the power law $\Psi(s)=s^2/2$, we arrive at the
individual population flows given by

\begin{equation}
\notag \label{flow_gurtin}
 J_i(u_1,u_2)=a u_i \nabla U.
\end{equation}
This model was introduced by Gurtin and Pipkin \cite{gurtin84} and mathematically analyzed
by Bertsch et al. \cite{bertsch85,bertsch12}. As remarked in \cite{gurtin84}, when
considering a set of species with different characteristics, such as size, behavior with
respect to overcrowding, etc., it is natural to assume that instead of the total population
density $u_1+u_2$ the individual flows $J_i$ depend on a general linear combination of both
population densities, possibly different for each population. This assumption leads to the
following expressions for the flows:

\begin{equation}
\label{def:general_flow} J_i(u_1,u_2)=u_i  \nabla (a_{i1}u_1+a_{i2}u_2).
\end{equation}
A more general evolution problem which included the flows of this type has been analyzed in
\cite{galiano12b}. A finite element fully discretized scheme was used to prove the existence
of solutions under rather general assumptions on the data.

The present article addrresses the \emph{singular} case $a_{ij}=a_i$ for $i,j=1,2$. Due to
the loss of ellipticity of the diffusion matrix, this case is more complicated for the
study. One of the possible approaches consists in considering the \emph{contact-inhibition}
problem, see \cite{chaplain06}, assuming that the components of the solution are initially
segregated:

\begin{equation}
\label{ci.id} \operatorname{supp} u_{1}(x,0) \cup \operatorname{supp}
u_{2}(x,0)=\Omega,\quad \operatorname{supp}u_{1}(x,0)\cap
\operatorname{supp}u_{2}(x,0)=\Gamma_0,
\end{equation}
where $\Omega\subset \mathbb{R}^n$ is the problem domain and $\Gamma_0\in\Omega$ is a given
hypersurface. In the one-dimensional case $\Omega=(-L,L)$ and $\Gamma_0=x_c\in (-L,L)$. A
segregated solution $\mathbf{u}=(u_1,u_2)$ of the cross-diffusion system

\begin{equation}
\label{eq:cross-diffusion}
\partial_t\mathbf{u}=\operatorname{div}\left(\mathbf{u}\nabla ( A\cdot
\mathbf{u}^\perp)\right)+\mathbf{f}(\mathbf{u}),\quad \mathbf{f}=(f_1,f_2),
\end{equation}
with a $2\times 2$ matrix $A$, is a solution which possesses the following property:
$u_1\cdot u_2=0$ and $u_1+u_2>0$ everywhere in the problem domain (we tacitly assume here
that the solution is so regular that these conditions make sense). The problem of existence
of segregated solutions of the cross-diffusion system \eqref{eq:cross-diffusion} in the
singular case $a_{ij}=1$ for $i,j=1,2$ 
was studied by Bertsch et.al. in \cite{bertsch12}. It is proved that for suitable initial
data the Cauchy problem for system \eqref{eq:cross-diffusion} has a segregated solution. In
\cite{bertsch10} (see also \cite{bertsch85}), the existence of segregated solutions was
proved in the case $n=1$ for the system

\begin{equation}
\label{eq:cross-diffusion-1} u_{it}=a_i\left(u_i\phi_{x}(u_1+u_2)\right)+f_i(u,v),\quad
i=1,2,\quad a_i=const>0,
\end{equation}
in the rectangular domain $(-L,L)\times (0,T]$ under the zero-flux boundary conditions for
$u_1+u_2$ on the lateral boundaries. The proofs in \cite{bertsch10,bertsch12} rely on the
observation that the introduction of the new thought function $w:=u_1+u_2$ transforms
systems \eqref{eq:cross-diffusion}, \eqref{eq:cross-diffusion-1} into systems composed of a
parabolic equation for $w$ and a transport equation for the function $r:=u_2/w$ with the
velocity field defined by $\nabla w$. Apart from the possibility to show the existence of
segregated solutions, this method allowed the authors of \cite{bertsch10} to derive the
equation of motion of the curve $x=\zeta(t)$ separating the parts of the problem domain
where either $u_1>0$, or $u_2>0$. The question of uniqueness of the segregated solutions for
systems \eqref{eq:cross-diffusion}, \eqref{eq:cross-diffusion-1} was left open.

\section{Formulation of the problem and main results}

Let $\Omega\subset \mathbb{R}^{n}$ be a bounded domain. We consider the problem of finding
nonnegative functions $(u,v)$ satisfying the conditions

\begin{equation}
\label{eq:system}
\begin{cases}
& u_t=\operatorname{div}\left(a_+u\nabla (u+v)\right)+f_{+}(u,v)\quad \text{in
$D=\Omega\times (0,T]$},
\\
& v_t=\operatorname{div}\left(a_-v\nabla (u+v)\right)+f_{-}(u,v)\quad \text{in $D$},\quad
a_\pm=const>0,
\\
& \text{$u+v=h$ on $\partial\Omega\times (0,T]$},
\\
& \text{$u(x,0)=u_0(x)$, $v(x,0)=v_0(x)$ in $\Omega$}.
\end{cases}
\end{equation}
It is assumed that the initial data are smooth and segregated:

\begin{equation}
\label{eq:segregate}
\begin{cases}
& \text{$u_0\geq 0$ and $v_0\geq 0$, $u_0\cdot v_0=0$ in $\overline{\Omega}$},
\\
& \text{${C}^{-1}\leq u_0+v_0\leq C$ in $\Omega$, $\;C=const>1$},
\\
& \text{$u_0+v_0\in C^{2+\alpha}(\overline{\Omega})$}.
\end{cases}
\end{equation}
Moreover, we assume that the supports of $u_0$ and $v_0$ are separated by a smooth
simple-connected hypersurface $\Gamma_0$,

\[
\Gamma_0=\partial\overline{\{x\in \Omega:\,v_0(x)>0\}},\qquad \Gamma_0\cap \partial
\Omega=\emptyset,
\]
which means that the domain $\Omega$ is split into two parts: the annular domain $\Omega_+$,
bounded by $\partial\Omega$ and $\Gamma_0$ (where $v_0=0$, $u_0>0$), and its complement
$\Omega_-$ (where $u_0=0$, $v_0>0$).  The functions $f_{\pm}(q,r)$ are assumed to satisfy
the conditions

\begin{equation}
\label{eq:f}
\begin{cases}
& f_+(0,r)=0,\quad \text{$f_+(q,0)$ is locally Lipschitz-continuous for $q\geq 0$},
\\
& f_-(q,0)=0,\quad \text{$f_-(0,r)$ is locally Lipschitz-continuous for $r\geq 0$},
\end{cases}
\end{equation}
an example of admissible $f_{\pm}$ is furnished by the functions

\[
f_+(q,r)=q(\alpha_{+}-\beta_{+}q-\gamma_{+}r),\quad
f_-(q,r)=r(\alpha_{-}-\beta_{-}q-\gamma_{-}r),
\]
$\alpha_{\pm},\beta_{\pm},\gamma_{\pm}=const>0$. Our aim is to construct a segregated
solution of problem \eqref{eq:system}. To this end we consider the initial and boundary
value problem for function $w=u+v$. If problem \eqref{eq:system} admits a segregated
solution such that $u+v>0$ and $u\cdot v=0$ everywhere in $\overline{D}$, it is necessary
that $w$ satisfies the conditions

\begin{equation}
\label{eq:dirichlet}
\begin{cases}
& w_t=\operatorname{div}\left(a\, w\nabla w\right)+f(w)\quad \text{in $D=\Omega\times
(0,T]$},
\\
& \text{$w=h$ on $\partial\Omega\times (0,T]$},
\\
& \text{$w(x,0)=w_0:=u_0+v_0$ in $\Omega$}
\end{cases}
\end{equation}
with the coefficient $a$ and the right-hand side $f$ defined by

\begin{equation}
\label{eq:a} a=\begin{cases} a_+ & \text{if $u>0$},
\\
a_- & \text{if $v>0$},
\end{cases}
\qquad f(w)=\begin{cases} f_+(w,0) & \text{$u>0$},
\\
f_-(0,w) & \text{if $v>0$}.
\end{cases}
\end{equation}
Problem \eqref{eq:dirichlet} is regarded as the initial and boundary value problem for a
parabolic equation with discontinuous data. If there is a continuous in $D$ solution $w$,
and if there exists a continuous bijective transformation $\Gamma_0\mapsto \Gamma_t$ of the
initially given surface $\Gamma_0$, we may try to define a solution of the original problem
\eqref{eq:system} by the equalities

\[
w(x,t)=
\begin{cases}
v(x,t) & \text{in the domain $\Omega^-(t)$ bounded by $\Gamma_t$, $t\in [0,T]$},
\\
u(x,t) & \text{in the complement $\Omega^+(t)$ of $\Omega^{-}(t)$ in $\Omega$}.
\end{cases}
\]

\begin{definition}
\label{def:reform} A pair $(w,\Gamma)$ is called weak solution of problem
\eqref{eq:dirichlet} if

\begin{enumerate}
\item $\Gamma$ is a $C^{1}$ hypersurface, the mapping $\Gamma_0\mapsto \Gamma_t=\Gamma\cap
\{t=const\}$ is a bijection for $t\in [0,T]$,

\item $\forall\;t\in [0,T]$ the surface $\Gamma_t$ is the common boundary of the domains
$\Omega^{\pm}(t)$,

\[
\Omega=\Omega^+(t)\cup \Gamma_t\cup \Omega^-(t),
\]
where $\Omega^{+}(t)$ is an annular domain bounded by $\partial\Omega$ and $\Gamma_t$,
$\Omega^{-}(t)$ is the complement of $\Omega^{+}(t)$ in $\Omega$,

\item  $w\in C^{0}(\overline{D})\cap L^{2}(0,T;H^{1}(\Omega^{\pm}(t)))$,

\item for every $\phi(x,t)\in C^{1}(\overline{D})$, such that $\phi(x,T)=0$, $\phi=0$ on
$\partial\Omega\times [0,T]$,

\begin{equation}
\label{eq:def} \int_{D}\left(w\phi_t-a w \nabla w\cdot \nabla
\phi+f(w)\,\phi\right)\,dxdt+\int_{\Omega}w_0\phi(x,0)dx=0.
\end{equation}

\end{enumerate}
\end{definition}

To construct a solution of problem (\ref{eq:dirichlet}) we proceed in two steps. The first
step consists in the direct construction of the surface $\Gamma$ and the corresponding
solution $w$ in a vicinity of $\Gamma$. This is done by means of a special coordinate
transformation similar to introduction of a system of Lagrangian coordinates frequently used
in continuum mechanics. Once the local solution is constructed, we continue it to the rest
of the problem domain and then check that this continuation is the thought solution of
problem \eqref{eq:dirichlet}.

\begin{theorem}[Local in time existence-1]
\label{th:lagrange} Let conditions \eqref{eq:segregate}, \eqref{eq:f} be fulfilled. Assume
that the data of problem \eqref{eq:dirichlet} satisfy the following conditions:

\begin{enumerate}
\item $\partial\Omega,\,\Gamma_0\in C^{2+\alpha}$, $w_{0}\in C^{2+\alpha}(\Omega)$ with some
$\alpha\in (0,1)$,

\item $\Gamma_0$ is a level surface of $w_0$, 

\item $h(x,t)>0$ on $\partial\Omega\times [0,T]$, $h(x,t)$ and $w_0(x)$ satisfy the
first-order compatibility conditions on $\partial\Omega\times \{t=0\}$.
\end{enumerate}
Then for every $\Phi(t)\in C^{1}[0,T]$

\begin{enumerate}
\item  there exists $T^\ast\leq T$ such that in the cylinder $\Omega\times (0,T^\ast]$
problem \eqref{eq:dirichlet} has a solution $w(x,t)$ in the sense of Definition
\ref{def:reform}, which satisfies the condition  $w=\Phi(t)$ on $\Gamma_t$,

\item the solution $w$ represents the segregated solution $(u,v)$ of system
\eqref{eq:system}: $w=u+v$, $u\equiv 0$ in $\Omega^-(t)\times [0,T^\ast]$, $v\equiv 0$ in
$\Omega^+(t)\times [0,T^\ast]$.
\end{enumerate}
\end{theorem}

The method of construction allows us to present the surface $\Gamma$ explicitly and to
derive the equation of motion of $\Gamma_t$, which is similar to the Darcy law in filtration
theory.

\begin{theorem}[The interface equation]
\label{th:Darcy} Under the conditions of Theorem \ref{th:lagrange} there exists an annular
domain $\omega^{+}(0)$, bounded by $\Gamma_0$ and a smooth hypersurface
$\partial\omega^{+}(0)$, $\partial\omega^{+}(0)\cap \partial\Omega=\emptyset$,
$\partial\omega^{+}(0)\cap \Gamma_0=\emptyset$, and a function $U(y,t)$ such that

\[
U\in W^{4}_{q}(\omega^{+}(0)\times [0,T^\ast]),\quad U_t\in W^{2}_{q}(\omega^{+}(0)\times
[0,T^\ast]),\quad \text{$U(y,0)=0$ in $\overline{\omega}^{+}(0)$}
\]
with some $q>n+2$, and $\Gamma$ is parametrized by the equalities

\[
\Gamma=\{(x,t):\,x=y+\nabla U(y,t),\,y\in \Gamma_0\},\quad t\in [0,T^\ast].
\]
Moreover, the velocity of advancement of the surface $\Gamma_t$ in the normal direction
$\mathbf{n}_x$ is defined by the equation

\begin{equation}
\label{eq:Darcy} \mathbf{v}\cdot \mathbf{n}_x=(-a_+\nabla u+\nabla p)\cdot
\mathbf{n}_x|_{\Gamma_t}\quad (\text{the modified Darcy law}),
\end{equation}
where $p$ is a solution of the elliptic equation

\[
\begin{cases}
& \text{$\operatorname{div}\left(u\nabla p\right)=f_{+}(u)$ in
$\omega^{+}(t)=\{x\in\Omega:\,x=y+\nabla U,\,y\in \omega^{+}(0)\}$},
\\
& \text{$p=0$ on $\Gamma_{t}$ and $\partial\omega^{+}(t)$}.
\end{cases}
\]
\end{theorem}

\begin{corollary}
\label{cor:indep} The components $u$ and $v$ of the solution $w=u+v$ to problem
\eqref{eq:dirichlet} constructed in Theorem \ref{th:lagrange} can be characterized in the
following way:

\begin{enumerate}
\item $u,\,v\in L^{\infty}(D)$, $u\geq 0,\,v\geq 0$ in $D$,

\item $u+v\in C^{0}(\overline{D})$, $u+v\in L^{2}(0,T;H^{1}(\Omega))$,

\item for every test-function $\phi\in C^{1}(\overline{D})$, $\phi(x,T)=0$, $\phi=0$ on
$\partial\Omega\times [0,T]$,

\begin{equation}
\label{eq:ind-1} \int_{D}\left(u\phi_t-a_+ u \nabla (u+v)\cdot \nabla
\phi+f_+(u)\,\phi\right)\,dxdt+\int_{\Omega}u_0\phi(x,0)dx=0, \end{equation}

\begin{equation}
\label{eq:ind-2} \int_{D}\left(v\phi_t-a_- v \nabla (u+v)\cdot \nabla
\phi+f_-(v)\,\phi\right)\,dxdt+\int_{\Omega}v_0\phi(x,0)dx=0
\end{equation}
\end{enumerate}
(cf. with Definition 3.1 in \cite{bertsch12}). The proof of this assertion is given in the
end of Section \ref{sec:proofs}.
\end{corollary}

\begin{theorem}[Nonuniqueness]
\label{th:nonuniqueness} Under the conditions of Theorem \ref{th:lagrange} the segregated
solution of probem \eqref{eq:system} is not unique.
\end{theorem}

The assertion of Theorem \ref{th:nonuniqueness} is an immediate byproduct of Theorem
\ref{th:lagrange}. Indeed: given $u_0$, $v_0$ and a level surface $\Gamma_0$ of the function
$w_0=u_0+v_0$, for every smooth $\Phi(t)$ such that $\Phi(0)=w_0|_{\Gamma_0}$ we obtain a
new solution of problem \eqref{eq:system} corresponding to the same initial data and
satisfying the condition $w=\Phi(t)$ on $\Gamma_t$.

The assumptions that $w=\Phi(t)$ on $\Gamma_t$ and that $\Gamma_0$ is a level surface of
$w_0=u_0+v_0$ are not essential for the proof of Theorem \ref{th:lagrange} and were included
in order to make evident nonuniqueness of segregated solutions of problem
\eqref{eq:dirichlet}.

\begin{theorem}[Local in time existence-2]
\label{th:lagrange-1} Let conditions \eqref{eq:segregate}, \eqref{eq:f} be fulfilled. Assume
that the data of problem \eqref{eq:dirichlet} satisfy the following conditions:
\begin{enumerate}
\item $\partial\Omega,\,\Gamma_0\in C^{2+\alpha}$, $w_{0}\in C^{2+\alpha}(\Omega)$ with some
$\alpha\in (0,1)$, \item $h(x,t)>0$, $h$ and $w_0$ satisfy the first-order compatibility
conditions on $\partial\Omega\times \{t=0\}$.
\end{enumerate}
Then there exists $T^\ast\leq T$ such that in the cylinder $\Omega\times (0,T^\ast]$ problem
\eqref{eq:dirichlet} has a solution in the sense of Definition \ref{def:reform}. The
solution $w$ of problem \eqref{eq:dirichlet} represents the segregated solution $(u,v)$ of
system \eqref{eq:system}: $w=u+v$, $u\equiv 0$ in $\Omega^-(t)\times [0,T^\ast]$, $v\equiv
0$ in $\Omega^+(t)\times [0,T^\ast]$. Moreover, for the interface of the constructed
solution Theorem \ref{th:Darcy} holds.
\end{theorem}

\begin{remark}
It is worth noting here that the choice of the Dirichlet boundary condition in
\eqref{eq:system} is mostly the question of convenience. The assertions of Theorems
\ref{th:lagrange}-\ref{th:lagrange-1} remain true if the boundary condition in
\eqref{eq:system} is substituted by any other condition which allows one to guarantee that
the auxiliary problem  \eqref{eq:A} below has a regular solution. In particular, we may pose
the no-flux conditions for $u+v$ on $\partial\Omega\times [0,T]$.
\end{remark}

The proofs of the main results are based on a special nonlocal coordinate transformation
which is similar to introduction of the system of Lagrangian coordinates in continuum
mechanics. The change of independent variables allows us to reduce the construction of the
moving boundary $\Gamma$ (the interface) to a problem posed in a time-independent domain. We
follow the ideas of \cite{Diaz-Shmarev-2009-1,Diaz-Shmarev-2009-2}, see also
\cite{Shm-NA-2003,Shm-2004,Shm-Vazquez-1996} where the method of Lagrangian coordinates was
applied to the study of free boundary problems for nonlinear parabolic equations with
degeneracy on the interface.

Organization of the paper. In Section \ref{sec:Lagrange} we introduce a local system of
Lagrangian coordinates. In the new coordinate system the problem of finding the surface
$\Gamma$ and the solution of problem \eqref{eq:dirichlet} in a vicinity of $\Gamma$
transforms into an equivalent problem posed in a time-independent cylinder. In the new
formulation the interface $\Gamma$ becomes a vertical surface. The new problem is a system
of nonlinear evolution equations which is solved in Section \ref{sec:Lagr-+}. In Section
\ref{sec:proofs} we give the proofs of the main theorems. Finally in Section
\ref{sec:special} we give an account of the available results on the problems of the type
\eqref{eq:cross-diffusion-1} without the contact inhibition assumption and present some
results on the numerical simulation of solution to system \eqref{eq:system} which correspond
to the segregated initial data.

\section{Local system of lagrangian coordinates}
\label{sec:Lagrange} Let us consider the following auxiliary problem: to find a strictly
positive function $w(x,t)$, a family of annular domains $\{\omega^\pm(t)\}_{t>0}$, and the
surface

\[
\Gamma=\bigcup_{t>0}\Gamma_t, \qquad
\Gamma_t=\overline\omega^{+}(t)\cap \overline\omega^{-}(t),
\]
satisfying the conditions

\begin{equation}
\label{eq:aux-Lagrange}
\begin{cases}
& \partial_t w-\operatorname{div}(a\, w\, \nabla w) =f(w)\quad
\text{in $\mathcal{C}^\pm=\bigcup_{t>0}\omega^\pm(t)$},
\\
& [w]|_{\Gamma_t}=0, 
\\
& \text{$w(x,0)= w_{0}(x)$ in $\omega^\pm(0)$},
\\
& \displaystyle\int_{\omega^\pm(t)}w(x,t)\,dx=\int_{\omega^\pm(0)}w_{0}(x)\,dx \quad
\forall\,t\in (0,T]
\end{cases}
\end{equation}
Here and throughout the rest of the paper the symbol $[\phi]_{\gamma}$ means the jump of the
function $\phi$ across the surface $\gamma$. The surface $\Gamma_0$ is the common boundary
of the annular domains $\omega^{\pm}(0)$. The exterior boundary of $\omega^+(0)$ is denoted
by $\partial\omega^+(0)$, $\partial\omega^-(0)$ stands for the interior boundary of
$\omega^-(0)$. Notice that problem (\ref{eq:aux-Lagrange}) includes three unknown
boundaries: the interface $\Gamma$ and $\bigcup\limits_{t>0}\partial\omega^{\pm}(t)$.

We will use the notations $\omega(t)=\omega^+(t)\cup \Gamma_t\cup \omega^-(t)$ and
$\mathcal{C}=\mathcal{C}^+\cup \mathcal{C}^-$.

\begin{definition}
\label{def:div-div} A pair $(w,\,\mathcal{C})$ is called weak
solution of problem \emph{(\ref{eq:aux-Lagrange})} if

\begin{itemize}
\item[(i)]  $w\in C^{0}(\overline{\mathcal{C}})\cap L^{2}(0,T;H^{1}(\omega(t)))$,

\item[(ii)] $\forall \, \phi\in C^{1}(\overline{\mathcal{C}})$, such that $\phi(x,T)=0$ and
$\phi=0$ on $\partial\omega^\pm(t)\times [0,T]$,

\begin{equation}
\label{eq:def-1}
\begin{split}
\int_{\mathcal{C}} (w\,\phi_t-a\,w\,\nabla w\cdot \nabla \phi+\phi\,f)\,dxdt &
+\int_{\omega(0)}\phi(x,0)w_{0}\,dx=0.
\end{split}
\end{equation}
\end{itemize}
\end{definition}

\subsection{A coordinate transformation in a moving annular domain}

Let us consider the problem of defining the family of transformations $X(y,t):\,S(0)\mapsto
S(t)$ of an open annular set $S(0)\subset \mathbb{R}^{n}$ and a function $w(x,t)$ according
to the following conditions:

\begin{itemize}
\item[a)] for every $t>0$

\begin{equation}
\label{eq:diffeomorphism} \text{$X(y,t):\,\overline{S(0)}\mapsto \overline{S(t)}\subset
\mathbb{R}^{n}$ is a diffeomorphism},
\end{equation}
that is $S(t)=X(S(0),t)$, $S(0)=X^{-1}(S(t),t)$,
$\partial S(t)=X(\partial S(0),t)$,

\item[b)] the deformation of $S(t)$ is governed by the differential equation

\begin{equation}
\label{eq:div}
\begin{cases}
& \operatorname{div}_x\,(w\,(X_t(y,t)-\mathbf{v}(X(y,t),t)))=0\quad \text{for a.e.}\,y\in
S(0),\, t>0,
\\
& X(y,0)=y\in S(0),
\end{cases}
\end{equation}
with a given vector-field $\mathbf{v}(x,t):\,S(t)\times [0,T]\mapsto \mathbb{R}^{n}$ in the
sense that for every $\phi\in C^{1}(0,T;C_{0}^{1}(S(t)))$

\[
\int_{S(t)}w\nabla\phi\cdot (X_t(y,t)-\mathbf{v}(X(y,t),t))\,dx =0,\quad t>0,
\]

\item[c)] for every subset $\sigma(0)\subset S(0)$ its image $\sigma(t)$ at the instant
$t\geq 0$ is connected with the function $w(x,t)$ by the formula

\begin{equation}
\label{eq:conserve-2}
\int_{\sigma(0)}w(x,0)\,dx=\int_{\sigma(t)}w(x,t)\,dx.
\end{equation}
\end{itemize}
Let $J$ be the Jacobian matrix of the mapping $y\to X(y,t)$, $|J|\not=0$ because of
(\ref{eq:diffeomorphism}). By agreement we always denote

\[
\widetilde g(y,t)=\left.g(x,t)\right|_{x=X(y,t)},
\]
so that $\widetilde w(y,t)\equiv w[X(y,t),t]\equiv w(x,t)$. Take an arbitrary set
$\sigma(0)\subseteq S(0)$ and denote $\sigma(t)=X(\sigma(0),t)$. For a.e. $t>0$

\begin{equation}
\label{eq:conserv-1}
\begin{split}
0 & =\dfrac{d}{dt}\left(\int_{\sigma(t)}w(x,t)\,dx\right)
 \\
 &
 =\dfrac{d}{dt}\left(\int_{\sigma(0)}\widetilde
w(y,t)\,|J|\,dy\right) =\int_{\sigma(0)}\dfrac{d}{dt}\left(\widetilde
w(y,t)\,|J|\right)\,dy,
\end{split}
\end{equation}
provided that $|J|$ is continuous as a function of $y$. Since $\sigma(0)$ is arbitrary and
$|J(y,0)|=1$, it is necessary that

\begin{equation}
\label{eq:conserve-Lagrange} \widetilde w(y,t)\,|J(y,t)|=w(y,0)\quad \text{for a.e. $y\in
S(0)$, $t>0$}.
\end{equation}

\begin{lemma}
\label{le:div-div} Assume that
\begin{enumerate}

\item $X$ satisfy \emph{(\ref{eq:diffeomorphism})}, $|J(y,t)|\in C^{0}(S(0))$ and
$|J(y,t)|\not=0$ in $S(0)$ for a.e. $t\in (0,T)$,

\item equations \emph{(\ref{eq:div})} and \emph{(\ref{eq:conserve-Lagrange})} are fulfilled
a.e. in the cylinder $S(0)\times (0,T)$,

\item $\widetilde w\cdot(X_t-\mathbf{v}(X,t))\in (L^{\infty}(S(0)))^{n}$ for a.e. $t\in
(0,T)$,

\item $\partial S(t)\in Lip$ for a.e. $t\in (0,T).$
\end{enumerate}
Then the function $w(x,t)|_{x=X(y,t)}=\widetilde w(y,t)$ defined  by
\emph{(\ref{eq:conserve-Lagrange})} satisfies the conditions

\begin{equation}
\label{eq:lagr}
\begin{cases}
& w_t-\operatorname{div}\,(w\,\mathbf{v}(x,t))=0\quad \text{in $\mathcal{D}\equiv
\bigcup_{t\in (0,T)}X(S(0),t)$},
\\
& \text{$w(x,0)=w_0(x)$ in $S(0)$}
\end{cases}
\end{equation}
in the following sense: $\forall\,\phi\in C^{1}(0,T;C_{0}^{1}(\mathbb{R}^{n}))$,
$\phi(x,T)=0$, $\phi=0$ on $\partial S(t)\times [0,T]$,

\begin{equation}
\label{eq:1-1}
\begin{split}
\int_{S(0)}w_0(x)\,\phi(x,0)\,dx & +\int_{\mathcal{D}} w\,\left(\phi_t+\nabla_{x}\phi\cdot
\mathbf{v}(x,t)\right)\,dxdt =0.
\end{split}
\end{equation}
\end{lemma}

\begin{proof} By \cite[Th.2.2]{Frid} for a.e. $t>0$ the field
$F:=w\,(X_t-\mathbf{v}(X,t))$ has the normal traces on every Lipschitz-continuous surface in
$\overline{S(t)}$ and the Green-Gauss formulas hold: for every $\phi\in
C^{1}(0,T;C_{0}^{1}(S(t)))$

\[
0=\int_{S(t)}\phi\,\operatorname{div}\,F\,dx=-\int_{S(t)}\nabla\phi\cdot F\,dx.
\]
Let us denote $D_T=S(0)\times (0,T]$. Notice that for every test-function $\phi\in
C^{\infty}(0,T;C_{0}^{1}(S(t)))$, vanishing as $t=T$,

\[
\left.\dfrac{d}{dt}\phi(x,t)\right|_{x=X(y,t)}=\widetilde \phi_t(y,t)+\widetilde{\nabla_x
\phi(x,t)}\cdot X_t.
\]
Then

\[
\begin{split}
-\int_{S(0)}w_0(y)\,\phi(y,0)\,dy &
=\int_{0}^{T}\dfrac{d}{dt}\left(\int_{S(t)}w(x,t)\,\phi(x,t)\,dx\right)\,dt
\\
& = \int_{0}^{T}\dfrac{d}{dt}\left(\int_{S(0)}\widetilde w(y,t)\,\widetilde
\phi(y,t)\,|J|\,dy\right)\,dt
\\
& = \int_{D_T}\dfrac{d}{dt}\left(\widetilde w(y,t)\,\widetilde
\phi(y,t)\,|J|\right)\,dy dt
\\
& =\int_{D_T}\left[\dfrac{d}{dt}\left(\widetilde
w(y,t)|J|\right)\widetilde \phi+\widetilde
w\,(\widetilde{\phi_t}+\widetilde{\nabla_x\phi}\cdot
X_{t})\,|J|\right]\,dydt.
\end{split}
\]
Using (\ref{eq:conserve-Lagrange})  we obtain

\begin{equation}
\label{eq:partial}
\begin{split}
- \int_{S(0)}w_0(y)\,\phi(y,0)\,dy & =\int_{D_T}\widetilde
w\,(\phi_t+\widetilde{\nabla_x\phi}\cdot X_{t})\,|J|\,dydt
\\
& =\int_{D_T}\widetilde
w\,\left(\widetilde{\phi_t}+\widetilde{\nabla_{x}\phi}\cdot
\mathbf{v}(X,t)\right)\,|J|\,dydt
\\
& = \int_{0}^{T}\int_{S(t)} w\,\left(\phi_t+\nabla_{x}\phi\cdot
\mathbf{v}(x,t)\right)\,dxdt.
\end{split}
\end{equation}
\end{proof}

\begin{theorem}
\label{th:interface} Assume that the domain $\omega(0)$ is split into two annular domains
$\omega^{\pm}(0)$ by the Lipschitz-continuous surface $\Gamma_0$ such that $\Gamma_0\cap
\partial\omega^\pm(0)=\emptyset$. If the conditions of Lemma \ref{le:div-div} are fulfilled in
each of the domains $\omega^{\pm}(0)$ and if

\begin{enumerate}
\item[(i)] $\lim\limits_{\omega^{+}(0)\ni y\to y_0\in
\Gamma_0}|J(y,t)|=\lim\limits_{\omega^{-}(0)\ni y\to y_0\in \Gamma_0}|J(y,t)|\quad
\forall\,y_0\in \Gamma_0,\;t\in [0,T]$,

\item[(ii)] $\mathbf{v}\in C^0(\overline{\omega^{+}(t)\cup \omega^-(t)})\times [0,T];
\,\mathbb{R}^{n})$,
\end{enumerate}
then $w(x,t)$ defined by \eqref{eq:conserve-Lagrange} satisfies conditions \eqref{eq:lagr}
in the sense of \eqref{eq:1-1}.
\end{theorem}

\begin{proof}
By Lemma \ref{le:div-div} problem \eqref{eq:lagr} has solutions $w^{\pm}$ in each of the
domains $\mathcal{C}^{\pm}$. By virtue of condition (ii) the images of the surface
$\Gamma_0$ under the mappings $X^{+}$ and $X^{-}$ coincide, which means that
$\Gamma_t=\overline{\mathcal{C}}^+\cap \overline{\mathcal{C}}^+$. The function
$w(x,t)=\widetilde w(y,t)$ defined by \eqref{eq:conserve-Lagrange} in each of the domains
$\mathcal{C}^{\pm}$ is continuous across the surface $\Gamma_t$ because of assumption (i).
Finally, to get \eqref{eq:lagr} we gather relations \eqref{eq:partial}, corresponding to the
domains $\omega^{\pm}(0)$.
\end{proof}

Theorem \ref{th:interface} will be used in the proof of Theorem \ref{th:lagrange-1}. In the
proof of Theorem \ref{th:lagrange} we rely on the following version of Theorem
\ref{th:interface}.

\begin{theorem}
\label{th:interface-normal} The assertion of Theorem \ref{th:interface} remains true if
condition {\rm (ii)} is substituted by the conditions

\begin{itemize}
\item[(iii)]$\quad w(x,t)=\Phi(t)$ on $ \Gamma_t$, $\qquad [\mathbf{v}\cdot
\mathbf{n}_{x}]_{\Gamma_t}=0$,
\end{itemize}
where $\mathbf{n}_{x}$ denotes the unit normal vector directed inward $\omega^{-}(t)$, and
$\Phi(t)$ is a given strictly positive function.
\end{theorem}

\begin{proof}
The assertion follows from Lemma \ref{le:div-div}: although the tangential component of the
velocity is no longer continuous across $\Gamma_t$, the assumption $w=\Phi(t)$ on $\Gamma_t$
provides continuity of the flux $w\,(\mathbf{v}\cdot \mathbf{n}_x)$ across $\Gamma_t$.
\end{proof}

\subsection{Potential flows} Let us now search for the fields $X(y,t)$ and $\mathbf{v}(X,t)$
in the potential form:

\begin{equation}
\label{eq:fields}
\begin{split}
& \text{$X(y,t)=y+\nabla_yU$ in $\omega^{\pm}(0)\times [0,T]$},
\\
& \text{$\mathbf{v}(x,t)=-a_{\pm}\,\nabla_x w + \nabla_x p$  in $\omega^{\pm}(0)\times
[0,T]$},
\\
& \text{$U=0$ on the parabolic boundaries of $\omega^{\pm}(0)\times [0,T]$},
\end{split}
\end{equation}
where $U(y,t)$ and $w(x,t)=\widetilde{w}(y,t)$ are scalar functions related by
\eqref{eq:conserve-Lagrange} and $p(x,t)$ is the new unknown. The parabolic boundary of a
cylinder means ``the lateral boundaries and the bottom". For every $\phi\in
C^{1}(0,T;C_{0}^{1}(\omega(t)))$,  $\phi(x,T)=0$, $\phi=0$ on $\partial\mathcal{C}$,

\begin{equation}
\label{eq:omega^+} \int_{\omega(0)}w_0\,\phi(x,0)\,dx+\int_{\mathcal{C}}
\left(w\,\phi_t-a\,w\,\nabla_{x}\phi\cdot\nabla_x
w\right)\,dx+\int_{\mathcal{C}}w\,\nabla_{x}\phi\cdot\nabla_x p\,dx=0.
\end{equation}
Let us take for $p$ a solution of the elliptic equation endowed with the Dirichlet boundary
conditions on $\partial\omega^{\pm}(t)$ and satisfying the additional condition on
$\Gamma_t$, which provides continuity of the flux $\Phi(t)\,\left(\mathbf{v}\cdot
\mathbf{n}_x\right)$ across the moving boundary:

\begin{equation}
\label{eq:ellip-p}
\begin{cases}
&  -\operatorname{div}_x\left(w\,\nabla_x p\right)=f(w)\quad \text{in $\omega^\pm(t)$},
\\
& \text{$p=0$ on $\partial\omega^{\pm}(t)$},
\\
&  [\nabla_xp\cdot \mathbf{n}_x]_{\Gamma_t}=[a\,\nabla_x w\cdot \mathbf{n}_x]_{\Gamma_t}.
\end{cases}
\end{equation}
Then for every smooth $\phi $, such that $\phi(x,T)=0$, $\phi=0$ on $\partial\mathcal{C}$,

\begin{equation}
\label{eq:omega^+-prim} \int_{\omega(0)}w_0\,\phi(x,0)\,dx+\int_{\mathcal{C}}
\left(w\,\phi_t-a\,w\,\nabla_{x}\phi\cdot\nabla_x w +f\,\phi\right)\,dx=0.
\end{equation}

Let us formulate the conditions for $U$ and $P=\widetilde{p}$ in the time-independent
annular cylinders

\[
Q^\pm_{T}=\omega^\pm(0)\times [0,T].
\]
Denote by $J$ the Jacobian matrix of the mapping $x=y+\nabla U$. Applying Lemma
\ref{le:div-div} we have that for every test-function $\phi(x,t)=\phi(X(y,t),t)=\widetilde
\phi(y,t)$, $\phi\in C^{0}(0,T;C^{1}_{0}(\omega(t)))$, $\phi=0$ on $\partial\omega(t)$,

\[
\begin{split}
0 &
=\int_{\omega^{\pm}(t)}\phi\,\operatorname{div}_x\,(w\,(\nabla_yU_t+a_-\widetilde{\nabla_x
w}-\widetilde{\nabla_{x}p })\,dx
\\
& =-\int_{\omega^{\pm}(0)}\left[w_0((J^{-1})^{2}\cdot \nabla_y\widetilde\phi)\cdot
\left(J\cdot \nabla_yU_t+a\, \nabla_y\left(w_0|J|^{-1}\right)- \nabla_y \widetilde
p\right)\right]\,dy.
\end{split}
\]
In particular, if $w_0,\,J_{ij}\in C^{\alpha}(\omega^{\pm}(0))$, and if $|J|$ is separated
away from zero, we may take for $\phi$ a solution of the problem

\[
\begin{cases} & \text{$\operatorname{div}_{y}\left(w_0(J^{-1})^{2}\cdot
\nabla_y\widetilde\phi\right)=\Delta_y\widetilde \psi$ in $\omega^{\pm}(0)$},
\\
& \text{$\phi=0$ on $\partial\omega^\pm(0)$}
\end{cases}
\]
with an arbitrary $\psi\in C^{0}(0,T;C^{1}_0(\omega(0)))$, whence

\[
\int_{\omega^{\pm}(0)}\psi\,\operatorname{div}_y\left(J\cdot \nabla_yU_t+a_+
\nabla_y\left(w_0|J|^{-1}\right)- \nabla_y \widetilde p\right)\,dy=0
\]
and

\begin{equation}
\label{eq:+} \begin{cases} &  \qquad \mathcal{L}(U,\widetilde p)\equiv
\operatorname{div}_y\left(J\cdot \nabla_yU_t+a \nabla_y\left(w_0|J|^{-1}\right)- \nabla_y
\widetilde p\right)=0\quad \text{in $Q^{\pm}_T$},
\\
& (a)\quad \text{$U=0$ on the parabolic boundaries of $Q_{T}^{\pm}$},
\\
& (b) \quad \text{$\Phi(t)|J|_{\Gamma_{0}}=\Phi(0)$ for all $t\in [0,T]$}.
\end{cases}
\end{equation}
The boundary condition \eqref{eq:+} (b) follows from \eqref{eq:conserve-Lagrange} and the
condition $w=\Phi(t)$ on $\Gamma_t$. (If we assume the conditions of Theorem
\ref{th:lagrange-1}, this condition is omitted). Proceeding in the same way we transform the
problem for $\widetilde p(y,t)=p(x,t)$ into the problem posed in the time-independent
domains $\omega^\pm(0)$:

\begin{equation}
\label{eq:++}
\begin{cases}
& \qquad \mathcal{M}(\widetilde p,U)\equiv
-\operatorname{div}_{y}\left(w_{0}(J^{-1})^{2}\nabla_y \widetilde p\right) +{f}
(w_0|J^{-1}|)|J|=0
\\
&\qquad \qquad \qquad
 \text{in $\omega^{\pm}(0)$ for a.e. $t\in (0,T]$},
\\
& (a) \quad \text{$\widetilde p=0$ on $\partial\omega^{\pm}(0)$},
\\
& (b)\quad [J^{-1}\cdot \nabla_y\widetilde p\cdot
\widetilde{\mathbf{n}}_x]_{\Gamma_0}=[a\,J^{-1}\cdot \nabla_y (w_0|J^{-1}|)\cdot
\widetilde{\mathbf{n}}_x]_{\Gamma_0}.
\end{cases}
\end{equation}
Condition \eqref{eq:++} (b) provides continuity of the normal component of the velocity
$\mathbf{v}$ across the moving boundary $\Gamma_t$.

\begin{theorem}
\label{th:lagr-euler} Let us assume that problem \eqref{eq:+}-\eqref{eq:++} has a solution
$(U,\widetilde{p})$ such that the conditions of Lemma \emph{\ref{le:div-div}} are fulfilled
with

\begin{equation}
\label{eq:Darcy-1} X=y+\nabla_y U\quad \text{ and}\quad \mathbf{v}=-a\,\nabla_x w+\nabla_x
p.
\end{equation}
Then the function $w(x,t)$ defined by the formulas

\begin{equation}
\label{eq:param-repr}
\begin{cases}
& \mathcal{C}^{\pm}=\left\{(x,t):\,x=y+\nabla_y U(y,t),\,(y,t)\in Q^{\pm}_T\right\},
\\
& w(x,t)={w_{0}(y)}{|J^{-1}|},
\end{cases}
\end{equation}
is a solution of problem \eqref{eq:aux-Lagrange} in the sense of Definition
\ref{def:div-div}. The moving boundaries of $\mathcal{C}$ and the interface $\Gamma$ are
parametrized by the equations

\[
x|_{\Gamma_t}=y|_{\Gamma_0}+\nabla_yU(y,t)|_{\Gamma_0},\qquad x|_{\partial\omega^{\pm}(t)}
=y|_{\partial\omega^{\pm}(0)}+\nabla_yU(y,t)|_{\partial\omega^{\pm}(0)}.
\]
\end{theorem}

The proof is an immediate byproduct of Theorem \ref{th:interface-normal}.

\subsection{Splitting the problems in the annular cylinders $Q_{T}^{\pm}$}

The next step is to split the nonlinear system \eqref{eq:+}-\eqref{eq:++} into two similar
systems in the annular cylinders $Q_{T}^{\pm}$ which can be solved sequentially. Let us
consider first the following problem for defining $(U^+,P^+)$:

\begin{equation}
\label{eq:split-U}
\begin{cases}
& \text{$\mathcal{L}(U^+,P^+)=0$ in $Q_{T}^{+}$},
\\
& \text{$\mathcal{M}(P^+,U^+)=0$ in $\omega^{+}(0)$},
\\
& \text{$U^+=0$ on the parabolic boundary of $Q_{T}^{+}$},
\\
(\ast) & \text{$|J|=\dfrac{\Phi(0)}{\Phi(t)}\equiv \Psi(t)$ on $\Gamma_0\times [0,T]$},
\\
& \text{$P^+=0$ on $\partial\omega^{+}(0)$ and $\Gamma_0$ for all $t\in [0,T]$}.
\end{cases}
\end{equation}
Let us assume that problem \eqref{eq:split-U} has a solution $(U^+,P^+)$ which satisfies the
regularity assumptions of Lemma \ref{le:div-div}. The function $P^+$ automatically satisfies
then the boundary condition \eqref{eq:++} (a) on the lateral boundaries of $Q_{T}^{+}$.
Given a pair $(U^+,P^+)$, we may formulate the problem for $(U^-,P^-)$ in $Q_{T}^{-}$, which
should include the conditions of zero jumps of density and the normal velocity across the
interface $\Gamma_t$. The problem in $Q_{T}^{-}$ is formulated as follows:

\begin{equation}
\label{eq:split-p}
\begin{cases}
& \text{$\mathcal{L}(U^-,P^-)=0$ in $Q_{T}^{-}$},
\\
& \text{$\mathcal{M}(P^-,U^-)=0$ in $\omega^{-}(0)$},
\\
& \text{$U^-=0$ on the parabolic boundary of $Q_{T}^{-}$},
\\
(\ast\ast) & \text{$|J^-|=\Psi(t)$ on $\Gamma_0\times (0,T]$},
\\
& \text{$P^-=0$ on $\partial\omega^{-}(0)$},
\\
& \text{$[J^{-1}\cdot \nabla_yP\cdot \widetilde{\mathbf{n}^{+}_x}]_{\Gamma_0}
=[a\,J^{-1}\cdot \nabla_y (w_0|J^{-1}|)\cdot \widetilde{\mathbf{n}^{+}_x}]_{\Gamma_0}$ for
$t\in [0,T]$},
\end{cases}
\end{equation}
where the upper index ``+" indicates that the corresponding magnitudes are already defined
by the functions $(U^+,P^+)$. By $\widetilde{\mathbf{n}^{+}_x}$ we denote the exterior
normal vector to the hypersurface $\Gamma_t$ parametrized by the formula $X=(y+\nabla
U^+)|_{y\in \Gamma_0}$. The vector $\widetilde{\mathbf{n}^{+}}_x$ is well-defined if
$\Gamma_0\in C^{2+\alpha}$ - see Remark \ref{rem:normal-moving} below. Once problems
\eqref{eq:split-U}, \eqref{eq:split-p} are solved, the functions

\[
U=\begin{cases} U^+ & \text{in  $Q^{+}_{T}$},
\\
U^- & \text{in  $Q^{-}_{T}$},
\end{cases}
\qquad \qquad \widetilde p=\begin{cases} P^+ & \text{in  $Q^{+}_{T}$},
\\
P^- & \text{in  $Q^{-}_{T}$}
\end{cases}
\]
define a solution of problem \eqref{eq:+}-\eqref{eq:++}.

Due to Theorem \ref{th:lagr-euler}, to solve problem \eqref{eq:aux-Lagrange} it suffices to
construct functions $U^\pm$, $P^\pm$ that satisfy the assumptions of Theorem
\ref{th:interface-normal} (or Theorem \ref{th:interface}).

\begin{remark}
\label{rem:1} Let the conditions of Theorem \ref{th:lagrange-1} be fulfilled. In order to
construct a solution of problem \eqref{eq:dirichlet} we omit condition $(\ast)$ in
\eqref{eq:split-U} and substitute condition $(\ast\ast)$ in \eqref{eq:split-p} by

\begin{equation}
\label{eq:J^-} \text{$|J^-|=|J^+|$ on $\Gamma_0\times [0,T]$}\qquad \text{(condition $(i)$
of Theorem \ref{th:interface})}
\end{equation}
Condition $(ii)$ of Theorem \ref{th:interface} has to be checked \emph{a posteriori}.
\end{remark}

\section{Problem in the annular cylinder $Q_{T}^{+}$}
\label{sec:Lagr-+}

Nonlinear problems similar to \eqref{eq:split-U}, \eqref{eq:split-p} were already studied in
\cite{Diaz-Shmarev-2009-1,Diaz-Shmarev-2009-2}. By this reason we confine ourselves to
presenting the main ideas of the proofs and omit the technical details.

We begin with problem \eqref{eq:split-U} posed in $Q_{T}^{+}$. To decouple the system of
equations for $U^+$ and $P^+$ we solve first the nonlinear equation $\mathcal{L}(U^+,P)=0$
considering $P$ as a given function from a suitable function space, and then solve the
linear elliptic equation $\mathcal{M}(P^+,U)=0$ with a given $U$. The solutions of these
equations generate an operator $\chi:\,(U,P)\mapsto (U^+,P^+)$. We show that the operator
$\chi$ has a fixed point, which is the sought solution of system \eqref{eq:split-U}.

\subsection{The function spaces}

Let $q>n+2$. We introduce the Banach spaces%

\[
\begin{split}
& \displaystyle \mathcal{Z}^{+} =\left\{  U:\;
\begin{array}
[c]{l} U \in W^{4}_{q}(Q_{T}^{+}),\;U_{t}\in W^{2}_{q}
(Q_{T}^{+}),\\
\mbox{$U=0$ on the parabolic boundary of $Q_{T}^{+}$}
\end{array}
\right\} , \\
& \mathcal{Y}^{+}=\left\{ f:\;f\in W^{2}_{q}(Q_{T}^{+})\right\},
\\
& \mathcal{X}^+ = \{\phi: \,\phi\in W^{2,1}_{q}(Q_{T}^{+}),\,\text{$\phi(y,0)=0$ in
$\omega^{+}(0)$}\}
\end{split}
\]
with the norms

\[
\|u\|^{(k)}_{q,Q_{T}^{+}}:=\|u\|_{W_{q}^{k}(Q_T^{+})}=\sum_{0\leq |\gamma|\leq
k}\|D_{y}^{\gamma}u\|_{q,Q_{T}^{+}},
\]

\[
\Vert U\Vert_{\mathcal{Z}^{+}}=\Vert U\Vert_{q,Q_{T}^{+}}^{(4)}+\Vert
U_{t}\Vert_{q,Q_{T}^{+}}^{(2)}, \qquad \Vert f\Vert_{\mathcal{Y}^{+}}=\Vert
f\Vert_{q,Q_{T}^{+}}^{(2)},\qquad
\|\phi\|_{\mathcal{X}^+}=\|\phi\|_{W^{2,1}_{q}(Q_{T}^{+})}.
\]
By $C^{\alpha}(Q^+_T)$, $\alpha\in (0,1)$, we denote the space of H\"older-continuous
functions equipped with the norm

\[
\big\langle v \big\rangle_{Q^+_T}^{(\alpha)} =\sup_{Q_{T}^{+}}|v|+\sup_{(x,t),(y,\tau)\in
Q^+_T}\dfrac{|v(x,t)-v(y,\tau)|}{|x-y|^{\alpha}+|t-\tau|^{\alpha/2}}.
\]
The embedding theorems yield that since $D^{2}_{{y}}U\in W^{2,1}_{q}(Q_{T}^{+})$ with
$q>n+2$, then

\begin{equation}
\label{eq:Hold} \forall\,U\in\mathcal{Z}^{+}\quad\sum_{|\gamma|=2,3}\big\langle
D^{\gamma}_{{y}}U\big\rangle^{(\alpha)}_{Q^{+}_{T}}\leq C \|U\|_{\mathcal{Z}^{+}}
\end{equation}
with some $\alpha\in(0,1)$ (see, e.g., \cite[Ch.2,~Lemma 3.3]{LSU}). Since $U({y},0)=0$, it
follows that

\begin{equation}
\label{eq:Holder-prim}\sum_{|\gamma|=2,3}\sup_{Q_{T}^{+}}|D^{\gamma }_{{y}}U| \leq
CT^{\alpha/2}\|U\|_{\mathcal{Z}^{+}}.
\end{equation}
Denote by $J$ the Jacobi matrix of the transformation $y\mapsto y+\nabla U$ and represent it
in the form $J=I+H(U)$, where $H(U)$ is the Hessian of $U$, $H_{ij}(U)=D^{2}_{ij}(U)$.
Estimate \eqref{eq:Holder-prim} allows us to choose $T$ so small that for every $U\in
\mathcal{Z}^+$, $\|U\|_{\mathcal{Z}^+}\leq 1$, the elements of the Jacobi matrix $J=I+H(U)$
and the Jacobian satisfy the estimates

\begin{equation}
\label{eq:Jacobian} \sup_{Q_{T}^{+}}|J_{ij}|\leq \delta_{ij}+C\,T^{\alpha/2},\qquad
\sup_{Q_{T}^{+}}\left||J|-1\right|\leq C\,T^{\alpha/2}\|U\|_{\mathcal{Z}^+}
\end{equation}
with an independent of $U$ constant $C$.

\subsection{The nonlinear parabolic problem}
Let $P\in \mathcal{Y^+}$ be given. Denote

\[
\mathcal{H}(U)=(\mathcal{H}_1(U),\mathcal{H}_{2}(U)), \quad
\begin{cases} \mathcal{H}_{1}(U)=\mathcal{L}(U,P) & \text{in $Q_{T}^{+}$},
\\
\mathcal{H}_{2}(U)=|J|-\Psi(t) & \text{on $\Gamma_0\times [0,T]$}.
\end{cases}
\]
The solution of the nonlinear problem

\begin{equation}
\label{eq:nonlinear-+} \mathcal{H}(U)=0,\quad U\in \mathcal{Z^+}
\end{equation}
is constructed by means of the modified Newton's method.

\begin{theorem}{\cite[Ch.~X]{Kolm-Fomin}}
\label{th:Newton} Let $\mathcal{X},\,\mathcal{Y}$  be Banach spaces. Assume that

\begin{enumerate}

\item the operator $\mathcal{H}(U):\,\mathcal{X}\mapsto \mathcal{Y}$ has the strong
differential $\mathcal{H}'(\cdot)$ in a ball $B_r(0)\subset \mathcal{X}$,

\item the operator $\mathcal{H}'(V)$ is Lipschitz-continuous in $B_r(0)$,

\[
\|\mathcal{H}^{\prime}(U_{1})-\mathcal{H}^{\prime}(U_{2})\|\leq L\,\|U_{1} -U_{2}\|,\qquad
L=const,
\]

\item there exists the inverse operator $\left[ \mathcal{H}'(0)\right] ^{-1}$ and

\[
\left\| \left[ \mathcal{H}^{\prime}(0)\right] ^{-1}\right\| =M,\qquad\left\| \left[
\mathcal{H}^{\prime}(0)\right] ^{-1}\langle\mathcal{H}(0)\rangle \right\| =\Lambda.
\]
\end{enumerate}
Then, if $\lambda=M\Lambda L<1/4$, the equation $\mathcal{H}(U)=0$ has a unique solution
$U^{\ast}$ in the ball $B_{\Lambda t_{0}}(0)$, where $t_{0}$ is the least root of the
equation $\lambda\,t^{2}-t+1=0$. The solution $U^{\ast}$ is obtained as the limit of the
sequence
\begin{equation}
\label{eq:iteration} U_{n+1} =U_{n}-[\mathcal{H}'(0)]^{-1}
\langle\mathcal{H}(U_{n})\rangle,\qquad U_{0}=0.
\end{equation}
\end{theorem}

Item (2) of Theorem \ref{th:Newton} means that the strong and weak defferentials of
$\mathcal{H}$ coincide and can be found by means of linearization of the operator
$\mathcal{H}$ at the initial state $U_0= 0$. Let us denote $J=I+H(U)$, where $H(U)$ is the
Hessian matrix of $U$, $H_{ij}(U)=D^{2}_{ij}(U)$. We have to compute

\[
\frac{d}{d\epsilon}\mathcal{H}(\epsilon U)=\frac{d}{d\epsilon}\operatorname{div}\left(
\epsilon \,({I} +\epsilon\,{H}(U))\nabla U_{t}+\nabla\left(  w_0 \,|{I}+\epsilon
{H}(U)|^{-1}\rangle-P\right) \right)
\]
at $\epsilon=0$. Since ${H}(U)$ is symmetric, for every fixed $(y,t)\in Q_{T}^{+}$ the
matrix $H(\epsilon\,U(y,t))$ is equivalent to the diagonal matrix with the eigenvalues
$\lambda_i$, $i=1,\ldots,n$. It follows that $|I+\epsilon
H(U)|=\prod\limits_{i=1}^{n}(1+\epsilon\lambda_i)$ and

\[
\left.\dfrac{d}{d\epsilon}\mathcal{H}_{2}(\epsilon
U)\right|_{\epsilon=0}=\dfrac{d}{d\epsilon}|J||_{\epsilon=0}
=\sum_{i=1}^{n}\operatorname{trace}\,H(U)=\Delta U.
\]
It is easy to see now that

\[
\left.\frac{d}{d\epsilon}\mathcal{H}_1(\epsilon\,U)\right|_{\epsilon=0}=\Delta
\left(U_t-a_+w_0\,\Delta U\right)
\]
and the linearized equation $\mathcal{H}'(0)(U)=\left(\Delta f,\phi\right)$ takes the form:
given $g\in \mathcal{Y}^+$, $\phi\in \mathcal{X}^+$, find a function $U\in \mathcal{Z}^+$
such that

\begin{equation}
\label{eq:par-linear} \begin{cases} &\Delta\left(U_t-a_+w_0\Delta U\right)=\Delta g\in
L^{q}(Q^{+}_T),
\\
& \left(\Delta U-\phi\right)|_{\Gamma_0\times [0,T]}=0.
\end{cases}
\end{equation}

\begin{lemma}
\label{th:par-linear} For every $(g,\phi)\in\mathcal{Y}^{+}\times \mathcal{X}^+$ problem
$(\ref{eq:par-linear})$ has at least one solution $U\in\mathcal{Z}^{+}$ satisfying the
estimate

\begin{equation}
\label{eq:est-tilde}\|U\|_{\mathcal{Z}^{+}}\leq C\,\left(\|g\|_{\mathcal{Y}
^{+}}+\|\phi\|_\mathcal{X^+}\right),\quad C\equiv C\left( n,q,\sup w_0,\inf
w_0,\|w_0\|^{(2)}_{q,\omega^+(0)}\right).
\end{equation}
\end{lemma}

\begin{proof} The proof follows \cite[Th.~9]{Diaz-Shmarev-2009-1} with obvious
modifications due to the form of the equation: instead of dealing with the heat equation now
we have to study problem \eqref{eq:par-linear} for a linear uniformly parabolic equation.
Let $U$ be a solution of the problem

\[
\begin{cases}
& U_t-a_+w_0\Delta U=  g+G \in L^{q}(Q_{T}^{+}),
\\
& \text{$U=0$ on the parabolic boundary of $Q_{T}^{+}$}
\end{cases}
\]
with a harmonic in $\omega^+(0)$ function $G$ to be defined. For every $g,\,G\in
L^{q}(Q^{+}_T)$ this problem has a unique solution $U\in W^{2,1}_{q}(Q_{T}^{+})$ which
satisfies the estimate

\begin{equation}
\|U\|_{W^{2,1}_{q}(Q_{T}^{+})}\leq
C\,(\|g\|_{q,Q_{T}^{+}}+\|G\|_{q,Q_{T}^{+}})\label{eq:reg-W}
\end{equation}
with a constant $C$ depending only on $q$, $n$, $\sup w_0$ and $\inf w_0$ (see
\cite[Ch.4,~Sec.9]{LSU}). Let us take for $G$ the solution of the Dirichlet problem

\[
\begin{cases}
& \text{$\Delta G(\cdot,t)=0$ in $\omega^+(0)$},
\\
& \text{$G(\cdot,t)+g(\cdot,t)=0$ on $\partial\omega^+(0)$},
\\
& \text{$G(\cdot,t)+g(\cdot,t)=a_+w_0(\cdot)\phi(\cdot, t)$ on $\Gamma_0$}.
\end{cases}
\]
(The boundary conditions are understood in the sense of traces). The function $G$ is
uniquely defined and satisfies the estimate

\[
\|G(\cdot,t)\|_{W^{2}_{q}(\omega^+(0))}\leq
C\,\left(\|g(\cdot,t)\|_{W_{q}^{2}(\omega^+(0))}+\|\phi(\cdot,t)\|_{W_{q}^{2}(\omega^{+}(0))}\right)\quad
\forall\,{\rm a.e.}\,t\in (0,T),
\]
which gives

\[
\|G\|_{W^{2}_{q}(Q_{T}^{+})}\leq
C\left(\|g\|_{W^{2}_{q}(Q_{T}^{+})}+\|\phi\|_{W^{2}_{q}(Q_{T}^{+})}\right).
\]
By construction

\[
\Delta (U_t-a_+w_0\Delta U-g)=\Delta G=0
\]
$U(y,0)=0$ and $U_t$ on $\partial\omega^{+}(0)\times [0,T]$. By the choice of $G$ the
function $g+G$ has zero trace on $\partial\omega^+(0)\times [0,T]$, while $g+G-a_+w_0\phi$
has zero trace on $\Gamma_0\times [0,T]$. By virtue of the equation for $U$ we have that
$\Delta U=0$ on $\partial\omega^+(0)\times [0,T]$ and $\Delta U=\phi$ on $\Gamma_0\times
[0,T]$.  It follows that $V=\Delta U$ solves the problem

\[
\begin{cases}
& V_t-a_+\Delta(w_0V)=\Delta g\in L^{q}(Q_{T}^{+})\quad \text{in $Q_{T}^{+}$},
\\
& \text{$V=0$ on $\partial\omega^+(0)\times [0,T]$},
\\
& \text{$V-\phi=0$ on $\Gamma_0\times [0,T]$}
\end{cases}
\]
and satisfies the estimate

\[
\|\Delta U\|_{W_{q}^{2,1}(Q_{T}^{+})}=\|V\|_{W_{q}^{2,1}(Q_{T}^{+})}\leq C \left(\|\Delta
g\|_{q,Q_{T}^{+}}+\|\phi\|_{W^{2}_{q}(Q_{T}^{+})}\right)
\]
with $C$ depending also on $\|w_0\|^{(2)}_{q,\omega^+(0)}$ (see \cite[Ch.4,~Sec.9]{LSU}).
Gathering this estimate with \eqref{eq:reg-W} we obtain \eqref{eq:est-tilde}.
\end{proof}

\begin{corollary}
\label{cor:constants}
\[
\begin{split}
& \left\| \left[ \mathcal{H}^{\prime}(0)\right] ^{-1}\right\| =M\leq C,
\\
& \left\| \left[ \mathcal{H}^{\prime}(0)\right] ^{-1}\langle\mathcal{H}(0)\rangle \right\|
=\Lambda\leq C\left(a_+T^{1/q}\|\Delta w_0\|_{q,\omega^+(0)}+\|P\|_{\mathcal{Y}^+}\right)
\\
&\qquad \qquad +T^{1/q}|\omega^{+}(0)|\left(\max_{[0,T]}|1-\Psi(t)|
+\max_{[0,T]}|\Psi'(t)|\right)
\end{split}
\]
with the constant $C$ from \eqref{eq:est-tilde}.
\end{corollary}

\begin{proof}
The estimates follow from \eqref{eq:est-tilde} and the equalities $
\mathcal{H}_1(0)=a_+\Delta w_0-\Delta P$, $\mathcal{H}_2(0)=1-\Psi(t)$.
\end{proof}
To prove the existence of a unique solution of the equation $\mathcal{H}(U)=0$ in
$\mathcal{Z}^+$ amounts to checking Lipshitz-continuity of the linearized operator

\[
\begin{split}
\mathcal{H}'(V)(U) &  =\operatorname{div}\Big( {H} (U)\,\nabla V_{t}+({I}+{H}(V))\,\nabla
U_{t}\Big)
\\
& \qquad \qquad \qquad -a_+\nabla\left( \operatorname{trace}\,\left[
({I}+{H}(V))^{-1}{H}(U)\right] \right),
\\
\mathcal{H}'_{2}(V)(U) & =\operatorname{trace} \left[(I+H(V))^{-1}H(U)\right],
\end{split}
\]
which can be done exactly as in \cite{Diaz-Shmarev-2009-1} with the use of formulas
\eqref{eq:Holder-prim}:

\begin{equation}
\label{eq:Lip}\left\| (\mathcal{H}'_{i}(V_{1})-\mathcal{H}'_{i}(V_{2})\right) \langle
U\rangle\|\leq L\|V_{1}-V_{2} \|_{\mathcal{Z}^{+}}\|U\|_{\mathcal{Z} ^{+}}.
\end{equation}

\begin{theorem}
\label{th:linear-aux-exist} Let $P\in W^{2}_{q}(Q_{T}^{+})$ with $q>n+2$ and $\Psi(t)\in
C^{1}[0,1]$. Then there exists $T_{\ast}\in(0,1)$ so small that $\lambda=M\,L\,\Lambda<1/4$
with the constants $\Lambda$, $M$ and $L$ from Corollary $\ref{cor:constants}$, and problem
\eqref{eq:nonlinear-+} has a unique solution

\begin{equation}
\label{eq:ball}U\in {B}_{r}\quad\text{with}\quad  r <2\,\Lambda.
\end{equation}
\end{theorem}

\begin{remark}
\label{rem:2} Under the conditions of Theorem \ref{th:interface} problem
\eqref{eq:nonlinear-+} transforms into the problem $\mathcal{H}_1(U)=0$, $U\in
\mathcal{Z}^+$, and the linearized problem \eqref{eq:par-linear} takes the form: find $U\in
\mathcal{Z}^+$ such that

\[
\Delta\left(U_t-a_+w_0\Delta U\right)=\Delta g\in L^{q}(Q^{+}_T).
\]
We may take for a solution the solution of \eqref{eq:nonlinear-+} with $\phi\equiv 0$. The
estimates of Corollary \ref{cor:constants} change in the obvious way,:

\[
\left\| \left[ \mathcal{H}^{\prime}(0)\right] ^{-1}\langle\mathcal{H}(0)\rangle \right\|
=\Lambda\leq C\left(a_+T^{1/q}\|\Delta w_0\|_{q,\omega^+(0)}+\|P\|_{\mathcal{Y}^+}\right).
\]
\end{remark}

\subsection{Linear elliptic problem}
Given $U\in \mathcal{Z}^+$, we consider now the equation $\mathcal{N}(P)\equiv
\mathcal{M}(P^+,U)=0$ in $Q_{T}^+$ under the homogeneous Dirichlet boundary conditions on
$\partial\omega^+(0)$ and $\Gamma_0$:

\begin{equation}
\label{eq:ellip-lin-1}
\begin{cases}
& \mathcal{N}(P)\equiv -\operatorname{div}_{y}\left(w_{0}(J^{-1})^{2}\nabla_y P\right)
+{f}_+ (w_0|J^{-1}|)|J|=0
\\
& \qquad \qquad \qquad \text{in $\omega^+(0)$, $t\in [0,T]$},
\\
& \text{$P=0$ on $\partial\omega^+(0)$ and $\Gamma_0$, $t\in [0,T]$}.
\end{cases}
\end{equation}

\begin{lemma}
\label{le:ellip-1} Let $w_0,\,D_{x_i}w_0\in L^{q}(\omega^{+}(0))$ and let $f_+$ be locally
Lipschitz-continuous. Then for every $U\in \mathcal{Z}^+$ with $\|U\|_{\mathcal{Z}^+}\leq 1$
problem \eqref{eq:ellip-lin-1} has a unique solution $P(\cdot,t)\in
W^{2}_{q}(\omega^{+}(0))$ such that

\begin{equation}
\label{eq:est-ell-1} \|P(\cdot,t)\|^{(2)}_{q,\omega^+(0)}\leq C
\|f_+(w_0|J|^{-1})\|_{q,\omega^+(0)}\quad \text{for a.e. $t\in (0,T)$}
\end{equation}
and

\begin{equation}
\label{eq:est-ellip-2} \|P\|_{\mathcal{Y}^{+}}\leq
C\,T^{1/q}\sup_{(0,T)}\|f_+(w_0|J|^{-1})\|_{q,\omega^+(0)}
\end{equation}
with a constant $C$ depending on $n$, $q$, $\sup w_0$, $\inf w_0$, $\|\nabla
w_0\|_{q,\omega^+(0)}$.
\end{lemma}

\begin{proof}
Using \eqref{eq:Jacobian} we choose $T$ be so small that $||J|-1|\leq \dfrac{1}{2}$, which
entails the inequalities

\[
\dfrac{1}{2}\leq |J|\leq \dfrac{3}{2},\qquad \dfrac{2}{3}\leq |J^{-1}|\leq 2\quad \text{in
$Q_{T}^{+}$}.
\]
Moreover, by virtue of \eqref{eq:Jacobian} $J$ is strictly positive definite for small $t$.
For every fixed $t$ the existence of a solution to problem \eqref{eq:ellip-lin-1} follows
immediately from the standard elliptic theory - see, e.g., \cite[Ch.~3, Sec.~5, 15]{LU}) or
\cite{Grisvard}. The second estimate follows upon integration of \eqref{eq:est-ell-1} over
the interval $(0,T)$.
\end{proof}

For $t=0$ problem \eqref{eq:ellip-lin-1} takes the form

\begin{equation}
\label{eq:ellip-initial}
\begin{cases}
& -\operatorname{div}_{y}\left(w_{0}\nabla_y P_0\right) +{f}_+ (w_0)=0\quad \text{in
$\omega^+(0)$},
\\
& \text{$P_0=0$ on $\partial\omega^+(0)$ and $\Gamma_0$}.
\end{cases}
\end{equation}

\begin{lemma}
\label{le:ellip-conv} Under the conditions of Lemma \ref{le:ellip-1}

\[
\|P(\cdot,t)-P_0\|^{(2)}_{q,\omega^+(0)}\leq C\,t^{\alpha/2}\|U\|_{\mathcal{Z}^+}.
\]
\end{lemma}

\begin{proof}
The function $P-P_0$ solves the problem

\[
\text{$-\operatorname{div}_{y}\left(w_{0}(J^{-1})^{2}\nabla_y (P-P_0)\right) =F$ in
$\omega^+(0)$},\qquad \text{$P-P_0=0$ on $\partial\omega^+(0)$ and $\Gamma_0$}
\]
with the right-hand side

\[
\begin{split}
F & =-({f}_+
(w_0|J^{-1}|)|J|-f_+(w_0))-\operatorname{div}_{y}\left(w_{0}(I-(J^{-1})^{2})\nabla_y
P_0\right)
\\
& =-({f}_+ (w_0|J^{-1}|)-f_+(w_0)) +f_+(w_0)(|J|-1)
-\operatorname{div}_{y}\left(w_{0}(I-(J^{-1})^{2})\nabla_y P_0\right)
\end{split}
\]
Since $f$ is locally Lipschitz-continuous, it follows from \eqref{eq:Holder-prim} and
\eqref{eq:Jacobian} that

\[
\|F(\cdot,t)\|_{q,\omega^{+}(0)}\leq C\,T^{\alpha/2} \|U\|_{\mathcal{Z}^+}
\]
with a constant $C$ depending also on the Lipshitz constant of $f(s)$ on the interval
$|s|\leq 2\sup w_0$. The required estimate follows now from \eqref{eq:est-ell-1}.
\end{proof}

\subsection{Solution of the nonlinear system \eqref{eq:split-U}} Following \cite{Diaz-Shmarev-2009-1}
we consider the sequences $\{U_k\}$, $\{P_k\}$ defined as follows: $U_0=0$, $P_0$ is the
solution of problem \eqref{eq:ellip-initial}, for every $k\geq 1$ $U_k$ is the solution of
\eqref{eq:split-U} with $P=P_k$, $P_{k+1}$ is the solution of problem \eqref{eq:ellip-lin-1}
with $U=U_k$. Gathering the estimates on the solutions of problems \eqref{eq:split-U},
\eqref{eq:ellip-lin-1} we find that independently of $k$

\[
\begin{split}
& \|U_k\|_{\mathcal{Z}^+}\leq C \left(a_+T^{1/q}\|\Delta
w_0\|_{q,\omega^+(0)}+\|P_k\|_{\mathcal{Y}^+}\right), \qquad \|P_k\|_{\mathcal{Y}^+}\leq
C\,R\,T^{1/q}
\end{split}
\]
with $R=\sup\{|f(s)|:\;|s|\leq 2 \sup w_0\}$, provided that $T$ is sufficiently small. It
follows that, up to subsequences,

\begin{equation}
\begin{split}
& \text{$U_{k}\rightharpoonup U$ in $\mathcal{Z}^{+}$},\quad \text{$P_{k}\rightharpoonup P$
in $\mathcal{Y}^+$},
\\
& \text{$D_iP_k\to D_iP$,$\quad D^{2}_{ij}U_{k}\to D^{2}_{ij}U$ in
$C^{\alpha^{\prime},\alpha^{\prime}/2}( \overline{D}_{T}^{+})$}
\end{split}
\label{eq:converge}%
\end{equation}
with some $\alpha'\in (0,1)$. Denote

\[
J_{k}=(I+H(U_{k})),\qquad \mathbf{v}_{k}=J_{k}^{-1}\nabla\left( a_+\,w_0|J_{k}|^{-1}
-P_{k}\right).
\]
By the method of construction

\[
\int_{\omega^+(0)}\eta\operatorname{div}\left( w_{0}\left(\,J_{k}\nabla U_{k,t}-
\mathbf{v}_k\right)\right) \,d{y} =0
\]
for every smooth test-function $\eta$. Passing to the limit as $k\to \infty$ we find that
$(U,P)$ is the solution of problem \eqref{eq:split-U}. Moreover, the constructed solution
possesses the regularity properties required in Lemma \ref{le:div-div}.

\begin{theorem}
\label{th:nonlinear-+} Let $w_0\in W^{2}_{q}(\omega^+(0))$ be strictly positive in
$\overline{\omega^+(0)}$, $f$ be Lipschitz-continuous on the interval $|s|\leq 2 \sup w_0$,
and let $\partial\omega^+(0),\Gamma_0\in C^{2+\beta}$ with some $\beta\in (0,1)$. There
exists $T^\ast$, depending on $\|w_0\|^{(2)}_{q,\omega^{+}(0)}$, $n$, $q$, $a_+$, $\beta$
and the Lipschitz constant of $f$ such that in the cylinder $\omega^+(0)\times (0,T^\ast]$
problem \eqref{eq:split-U} has a unique solution $U\in \mathcal{Z}^+$, $P\in \mathcal{Y}^+$.
\end{theorem}

\begin{remark}
\label{rem:normal-moving} The normal vector $\mathbf{n}_x$ is well-defined because
$\Gamma_0\in C^{2+\alpha}$ and $\mathbf{v}=J^{-1}\nabla\left( a_+\,w_0|J|^{-1} -P\right)$ is
continuous in $t$ due to \eqref{eq:Holder-prim} and Lemma \ref{le:ellip-conv}.
\end{remark}

By the method construction, the obtained solution satisfies all the conditions of Lemma
\ref{le:div-div} except bijectivity of the mappings $\partial\omega^+(0)\mapsto
X(\partial\omega^+(0),t)=\partial\omega^{+}(t)$, $\Gamma_0\mapsto X(\Gamma_0,t)=\Gamma_t$,
which has to be checked independently.

\begin{lemma}
\label{le:bijectivity-+} Under the conditions of Theorem \ref{th:nonlinear-+} the value of
$T^\ast$ can be chosen so small that for every points $y,\,z\in \omega^{+}(0), \,\Gamma_0$

\[
|X(y,t)-X(z,t)|\geq \mu\,|y-z|
\]
with an independent of $y,z$ constant $\mu\in (0,1)$.
\end{lemma}

\begin{proof}
Let us fix an arbitrary pair of points $y,z\in\Gamma_0$ and connect them by a
Lipschitz-continuous curve $l(y,z)\subset \omega^{+}(0)$. Since $\Gamma_0$ is smooth, we can
choose $l(y,z)$ in such a way that its length $|l(y,z)|$ satisfies the estimates
$\kappa_1|y-z|\leq |l(y,z)|\leq \kappa_2|y-z|$ with finite constants $\kappa_i$ depending
only on module of continuity of the parametrization of $\Gamma_0$.  By the definition

\[
X(y,t)-X(z,t)=(y-z)+\nabla (U(y,t)-U(z,t))=(y-z)+\int_{l(y,z)}\dfrac{d}{dl}(\nabla U)\,ds
\]
and by virtue of \eqref{eq:Holder-prim}

\[
|X(y,t)-X(z,t)|\geq |y-z|-\sum_{|\gamma|=2}\sup_{Q_{T}^{+}}|D^{\gamma}U||l(y,z)|\geq
|y-z|\left(1-C\,\kappa_2\,T^{\alpha/2}\right).
\]
\end{proof}


\subsection{Problem in the cylinder $Q^{-}_{T}$ and a local solution of the free-boundary problem}

To construct a solution of problem \eqref{eq:split-p} we follow the same scheme that was
used to find a solution of problem \eqref{eq:split-U}. The only difference is that now the
solution $P^-$ of the linear elliptic problem has to satisfy the Neumann boundary condition
on $\Gamma_0$. Let us define the function spaces $\mathcal{Z}^-$, $\mathcal{Y}^-$ ,
$\mathcal{X}^-$, where the upper index means that we consider the functions defined on
$Q_{T}^{-}=\omega^-(0)\times [0,T]$. Problem \eqref{eq:split-p} is split into the problems
for defining $U^-$ and $P^-$. The first step is to find a solution $U^{-}$ of the problem

\begin{equation}
\label{eq:split-p-}
\begin{cases}
& \text{$\mathcal{L}(U^-,P^-)=0$ in $Q_{T}^-$},
\\
& \text{$U^-=0$ on the parabolic boundary of $Q_{T}^{-}$},
\\
& \text{$|J^-|=\Psi(t)$ on $\Gamma_0\times (0,T]$}
\end{cases}
\end{equation}
with a given $P^-\in W_{q}^{2}(Q_{T}^-)$. The boundary condition for $|J^{-}|$ is
substituted \eqref{eq:J^-} in case of Theorem \ref{th:lagrange-1}. Repeating the proof of
Theorem \ref{th:linear-aux-exist} we arrive at the following assertion.

\begin{lemma}
\label{le:linear-prim} Let $P^-\in W^{2}_{q}(Q_{T}^{-})$ with $q>n+2$ and $\Psi(t)\in
C^{1}[0,1]$. Then there exists $T_{\ast}\in(0,1)$ so small that problem \eqref{eq:split-p-}
has a unique solution $U^-\in \mathcal{Z}^-$ such that $\|U^-\|_{\mathcal{Z^-}}\leq r'<1$
and $r'\to 0$ as $T_{\ast}\to 0$.
\end{lemma}

The second step is to solve the problem

\begin{equation}
\label{eq:P^-}
\begin{cases}
& \text{$\mathcal{M}(P^-,U^-)=0$ in $\omega^{-}(0)$},
\\
& \text{$P^-=0$ on $\partial\omega^{-}(0)$}, \quad \text{$(J^-)^{-1}\cdot \nabla_yP^-\cdot
\widetilde{\mathbf{n}^{+}_x}=S$ on $\Gamma_0$}
\end{cases}
\end{equation}
with given $U^{\pm}\in \mathcal{Z}^{\pm}$, $P^+\in W_{q}^{2}(Q_{T}^{+})$ and

\[
S=(J^+)^{-1}\cdot \nabla_yP^-\cdot \widetilde{\mathbf{n}^{+}_x}+ [a\,J^{-1}\cdot \nabla_y
(w_0|J^{-1}|)\cdot \widetilde{\mathbf{n}^{+}_x}]_{\Gamma_0}.
\]

\begin{lemma}
\label{le:elliptic-1} Let $U^{\pm}\in\mathcal{Z}^{\pm}$, $P^{+}\in W_{q}^{2}(Q_{T}^{+})$. If
$f_-$ is locally Lipschitz-continuous, then for a.e. $t\in(0,T)$ problem $(\ref{eq:P^-})$
has a solution $P(\cdot,t)\in W^{2}_{q}(\omega^{-}_{0})$ which satisfies the estimates

\begin{equation}
\label{eq:sec}\|P\|_{\mathcal{Y}^{-}}\leq
C\,\left(\|U\|_{\mathcal{Z}^{+}}+\|U\|_{\mathcal{Z}^{-}}+\|P^+\|_{\mathcal{Y}^+}
+T^{1/q}\max_{[0,T]}|\Psi'(t)|\right)
\end{equation}
with an absolute constant $C$.
\end{lemma}

\begin{proof}
The existence of a solution of problem \eqref{eq:P^-} satisfying \eqref{eq:sec} follows from
the classical elliptic theory - see, e.g., \cite[Ch.~3, Sec.~5-6, 15]{LU}) or
\cite{Grisvard}.
\end{proof}
Recall that in the case of Theorem \ref{th:lagrange-1} the corresponding estimate
\eqref{eq:sec} is independent of $\Phi(t)$.

The next step consists in checking the convergence of the iteratively defined sequences
$\{U_{k}^-\}$, $\{P_{k}^-\}$: $U_{0}^{-}=0$, $P^{-}_0$ is the solution of problem
\eqref{eq:P^-} with $U^-=0$, for every $k\geq 1$ $U^{-}_k$ is the solution of
\eqref{eq:split-p-} with $P=P^{-}_k$, $P^{-}_{k+1}$ is the solution of problem
\eqref{eq:P^-} with $U^{-}=U^{-}_k$. This is done exactly as in the proof of Theorem
\ref{th:nonlinear-+}.

\begin{lemma}
\label{le:nonlinear-} Let $w_0\in W^{2}_{q}(\omega^\pm(0))$ be strictly positive in
$\overline{\omega(0)}$, $f$ be Lipschitz-continuous on the interval $|s|\leq 2 \sup w_0$,
and let $\partial\omega^\pm (0),\Gamma_0\in C^{2+\beta}$ with some $\beta\in (0,1)$. There
exists $T^\ast$, depending on $\|w_0\|^{(2)}_{q,\omega(0)}$, $n$, $q$, $a_\pm$, $\beta$ and
the Lipschitz constant of $f$ such that problems \eqref{eq:split-U}, \eqref{eq:split-p} have
unique solutions $(U^{\pm},P^{\pm})\in \mathcal{Z}^{\pm}\times \mathcal{Y}^{\pm}$.
\end{lemma}

Finally, we repeat the proof of Lemma \ref{le:bijectivity-+} to ensure the bijectivity of
the mapping $y\mapsto X(t,t):=y+\nabla U$ for $y\in \overline{\omega}^{-}(0)$. The assertion
of Theorem \ref{th:lagr-euler} follows now if we define

\[
U=\begin{cases} U^+ & \text{in $Q_{T}^{+}$},
\\
U^- & \text{in $Q_{T}^{-}$},
\end{cases}
\qquad P=\begin{cases} P^+ & \text{in $Q_{T}^{+}$},
\\
P^- & \text{in $Q_{T}^{-}$}.
\end{cases}
\]

\section{Proofs of the main results}
\label{sec:proofs}

\subsection{Continuation to the rest of the cylinder. Proof of Theorem \ref{th:lagrange}}

Let us denote by $\Sigma^{\pm}$ the images of the surfaces $\partial^\pm\omega(0)$ under the
mapping $y\mapsto X(y,t)$. According to Theorem \ref{th:lagr-euler} the pair
$(w,\mathcal{C})$ defined by formulas (\ref{eq:param-repr}) is a solution of problem
(\ref{eq:aux-Lagrange}) in the sense of Definition \ref{def:div-div}. Let us take a smooth
simply connected surface $\gamma\subset \omega^{+}(0)$ such that $\gamma\cap
\Gamma_0=\emptyset$ and $\gamma\cap
\partial\omega^+(0)=\emptyset$. By continuity of the mapping $y\mapsto y+\nabla U$, there is
$T^+$ such that $\Sigma^{+}$ and $\Gamma_t$ do not touch the vertical surface
$S=\gamma\times [0,T^+]$, so that $S\subset \mathcal{C}^+$. Since $w_0>0$ in $\omega^+(0)$,
the function $w$ constructed in Theorem \ref{th:lagr-euler} is strictly positive in
$\mathcal{C}^+$ and is a weak solution of the uniformly parabolic equation. The local
regularity results for the solutions of uniformly parabolic quasilinear equations
\cite[Ch.~6, Sec.~4]{LSU} imply that $w\in C^{2+\beta,(2+\beta)/2}_{x,t}$ in a vicinity of
$S$. Let us set $\psi=w|_{S}\in C^{2+\beta,(2+\beta)/2}(S)$, denote by $\mathcal{A}$ the
annular cylinder with the lateral boundaries $\partial \Omega\times[0,T^+]$ and $S$, and
consider the following problem:

\begin{equation}
\label{eq:A} \left\{
\begin{array}
[c]{ll}
& u_{t}=\operatorname{div}(a_{+} u\,\nabla u)+f_+(u)\quad\mbox{in $\mathcal{A}$},\\
& \mbox{$u=\psi$ on $S$},\quad\mbox{$u=h$ on $\partial\Omega
\times\lbrack0,T^+]$},\\
& \mbox{$u({x},0)=w_{0}$ in $\overline{\mathcal{A}}\cap\{t=0\}$}.
\end{array}
\right.
\end{equation}
This problem has a unique solution $u\in C^{2+\beta,(2+\beta)/2}%
(\overline{\mathcal{A}})$, that is,

\[
\left.  D_{x}^{\kappa}D_{t}^{s}(u-w)\right\vert _{\Sigma}=0\quad \mbox{for
$0\leq|\kappa|+2s\leq2$}.
\]
The required continuation to the exterior of $\mathcal{C}^+$ is now given by the formula

\[
W= \left\{
\begin{array}
[c]{ll}
w({x},t) & \mbox{in $\mathcal{C}^+\setminus\mathcal{A}$},\\
u({x},t) & \mbox{in $\mathcal{A}$},
\end{array}
\right.
\]
The continuation from $\mathcal{C}^{-}$ is constructed likewise.

\subsection{Proof of Corollary \ref{cor:indep}} The proof if a byproduct of the proof of
Theorem \ref{th:lagrange}. Items (1)-(2) follows directly from Theorem \ref{th:lagrange}. By
Lemma \ref{le:div-div} $u$ is obtained as the solution of problem \eqref{eq:lagr} in the
moving domain $\mathcal{C}^{+}$ and then continued across the exterior boundary of
$\mathcal{C}^{+}$ up to the lateral boundary of $D$ by the solution of problem \eqref{eq:A}.
Recall that by construction $u$ satisfies the equation
\[
 u_t+\operatorname{div}(u\mathbf{v})=0 \quad \text{for a.e.}\,(x,t) \in \mathcal{C}^+
\]
with $\mathbf{v}=-a_+ \nabla u+\nabla p$ (see \eqref{eq:Darcy-1}). Let $S$ and $\mathcal{A}$
be the sets chosen in the proof Theorem \ref{th:lagrange},
$\mathcal{A}_0=\overline{\mathcal{A}}\cap \{t=0\}$. By Lemma \ref{le:div-div}, for every
$\phi\in C^{1}(\overline{Q}\setminus \mathcal{C}^-)$, $\phi=0$ on $\partial\Omega\times
[0,T]$, $u$ satisfies \eqref{eq:1-1}:

\[
\begin{split}
-\int_{S}\phi a_+u\nabla u\cdot \mathbf{n}^+\,dS & +\int_{\mathcal{C^{+}}\setminus
\mathcal{A}} u\,\left(\phi_t+\nabla_{x}\phi\cdot
\mathbf{v}(x,t)\right)\,dxdt
\\
& +\int_{\omega(0)\setminus\mathcal{A}_0}u_0(x)\,\phi(x,0)\,dx =0.
\end{split}
\]
Continuing $u$ to $\mathcal{A}$ by the classical solution $\widetilde u$ of problem
\eqref{eq:A} we have

\[
\begin{split}
\int_{S}\phi a_+u\nabla u\cdot \mathbf{n}^{+}\,dS & +\int_{\mathcal{A}}
\left(\widetilde{u}\phi_t+a_+\widetilde{u} \nabla_{x}\phi\cdot \nabla \widetilde u
-f_+(\widetilde{u})\phi\right)\,dxdt
\\
& +\int_{\mathcal{A}_0}u_0\,\phi(x,0)\,dx =0.
\end{split}
\]
Gathering these equalities and taking into account the definition of $p$, we obtain
\eqref{eq:ind-1}. Relation \eqref{eq:ind-2} follows by the same arguments.

\subsection{Proof of Theorem \ref{th:Darcy}} The assertion is an immediate byproduct of the method of
construction of the solution to problem \eqref{eq:dirichlet}.

\subsection{Proof of Theorem \ref{th:lagrange-1}} The assertion of Theorem
\ref{th:lagrange-1} will follow if we prove that the velocity given by formula
\eqref{eq:Darcy} is continuous on $\Gamma_t$. The normal component of velocity is continuous
on $\Gamma_t$ by the definition. Let us fix an arbitrary point $y_{0}\in\Gamma_0$ and denote
by $x^{+}_0=X^+(y_0,t)$ its image under the mapping $X^+=y+\nabla U^+$ in $Q_{T}^{+}$. By
the definition, for every $t>0$

\[
\mathbf{v}^+(x^{+}_{0},t)=X^{+}_{t}(y_0,t)=\nabla_y U_t^{+}(y_0,t).
\]
Let $\tau(y_0)$ be an arbitrary unit vector in the tangent plane to $\Gamma_0$ at the point
$y_0$. Since $U^{+}_t=0$ on $\Gamma_0\times [0,T]$, we have

\[
\text{$\nabla_yU^{+}_{t}(y_0,t)\cdot \tau(y_0)=0\quad $ and $\quad
\mathbf{v}^+(x_{0}^{+},t)\cdot \tau(y_0)=0$ for all $t>0$},
\]
which means that for all $t>0$ the direction on the velocity $\mathbf{v}^{+}(X^+(y_0),t)$
coincides with $\mathbf{n}_{x}(y_0)$. Repeating this argument we find that the direction of
$\mathbf{v}^{-}(X^-(y_0),t)$ is also given by $\mathbf{n}_{x}(y_0)$ for all $t>0$. Thus, the
images $X^{\pm}(y_0,t)$ of the point $y_0\in \Gamma_0$ move along the same line with the
direction vector $\mathbf{n}_{x}(y_0)$. Since $[\mathbf{v}(x_0)]\cdot \mathbf{n}_{x}(x_0)=0$
by construction, it is necessary that the tangent component of $\mathbf{v}$ is also
continuous at $x_0$: every tangent vector $\tau(y_0)$ can be represented in the form
$\tau(y_0)=\alpha \tau(x_0)+\beta\,\mathbf{n}_{x}(x_0)$ with $\alpha\not=0$ (for small $t$),
whence

\[
0=[\mathbf{v}(x_0)]\cdot \tau(y_0)=\alpha [\mathbf{v}(x_0)]\cdot
\tau(x_0)+\beta\,[\mathbf{v}(x_0)]\cdot\mathbf{n}_{x}(x_0)=\alpha [\mathbf{v}(x_0)]\cdot
\tau(x_0).
\]

\section{Special cases}
\label{sec:special} In this section, we review special cases of system \eqref{eq:system}
available in the literature. The first example concerns the possibility to construct a
solution assuming that neither the contact inhibition assumption \eqref{ci.id} on the
initial data is fulfilled, nor that the matrix $A$ in \eqref{eq:cross-diffusion} is positive
definite. The second example is an explicit solution that corresponds to specific initial
data generated by the self-similar Barenblatt solution of the porous medium equation.
Finally we provide examples of numerical simulations.

\subsection{The singular case without the contact-inhibition assumption}

Given a fixed $T>0$ and a bounded set $\Omega\subset\mathbb{R}^n$, with $\partial \Omega \in
C^{0,1}$, find $u_i:\Omega\times (0,T]=Q_T\to\mathbb{R}$, $i=1,2$, such that

\begin{equation}
\label{eq:pde}
\begin{cases}
& \partial_t u_i-\operatorname{div} J_i(u_1,u_2)=u_i F_i(u_1,u_2) \quad \text{in $Q_T$},
\\
& J_i(u_1,u_2)\cdot n =0 \quad \text{on $\Gamma_T= \partial\Omega\times (0,T]$},
\\
& u_{i}(\cdot,0)=u_{i0} \quad \text{in $\Omega$},
\end{cases}
\end{equation}
with the flows given by
\begin{align}
 J_i(u_1,u_2)=a u_i \nabla (u_1+u_2),   \label{def:flow2}
\end{align}
and the Lotka-Volterra terms of the special type

\begin{equation}
\label{LVbertsch} F_1(u_1,u_2)=1-u_1-\alpha u_2,\quad F_2(u_1,u_2)=\gamma(1-\beta u_1-u_2/k)
\end{equation}
with positive constants $\alpha$, $\beta$, $\gamma$ and $k$.
\begin{theorem}[\cite{bertsch12}]
\label{th:bertsch} For $i=1,2$, let $u_{i0}\in C^3(\bar \Omega)$ such that $u_{i0}\geq 0$
and $B_0\leq u_{10}+u_{20} \leq B_0^{-1}$, for some constant $B_0$. Then there exist a
solution $u_i \in C^{2,1}([0,\infty)\times\Omega)$ of \eqref{eq:pde} with $J_i$ given by
\eqref{def:flow2} and $F_i$ by \eqref{LVbertsch}.
\end{theorem}
The requirement of the strong regularity of the initial data is due to method of proof.
Initially, the following formally equivalent system is solved for $u=u_1+u_2$ and $v=u_1/u$:

\begin{equation}
\label{eq:pde12}
\begin{cases}
& \partial_t u-\operatorname{div} (u\nabla u)=G_1(u,v) \quad \text{in $Q_T$},
\\
& \partial_t v- \nabla u\cdot\nabla v = G_2(u,v)  \quad \text{in $Q_T$},
\\
& u\nabla u\cdot n =0 \quad \text{on $\Gamma_T$},
\\
& u_0(\cdot,0)=u_{10}+u_{20}, \quad v_0(\cdot,0)=u_{10}/u \quad \text{in $\Omega$},
\end{cases}
\end{equation}
with some smooth functions $G_1,~G_2$. The proof of existence of solutions  of
\eqref{eq:pde12} is based on the Schauder fixed point theorem. In order to obtain the
required compactness for the fixed point operator, the authors pass to the system of
Lagrangian coordinates related to the flow $-\nabla u(t,x)$, and claim the strong regularity
assumptions on the initial data.

A similar problem  was studied in \cite{galiano12b} under weaker assumptions on the initial
data and with a more general flow of the type

\begin{align}
 J_i(u_1,u_2)=a u_i \nabla (u_1+u_2) +b q u_i +c\nabla u_i .   \label{def:flow}
\end{align}
The existence was proved with a different method.

\begin{theorem}[\cite{galiano12b}]
\label{th.gs} \label{th:orig} Assume the following conditions: for $i=1,2$

\begin{enumerate}
\item the flows $J_i(u_1,u_2)$ are given by \eqref{def:flow} with constant $a>0$, $c\geq0$
and $b\in \mathbb{R}$, $q\in L^2(Q_T)$ and $\operatorname{div} q \geq 0$ a.e. in $Q_T$,

\item $u_{i0}\in L^\infty(\Omega)$ with $u_{i0}\geq0$, $u_0=u_{10}+u_{20} \in H^2(\Omega)$
with  $u_{0} > \delta$ for some constant $\delta>0$, $\nabla u_0 \cdot \mathbf{n} =0$ on
$\partial\Omega$ (the compatibility condition),

\item $F_1(u_1,u_2)=F_2(u_1,u_2)=F(u_1+u_2)$ with $F\in C^0(\mathbb{R}_+)$,  $\forall s\geq
0,~ F(\cdot,\cdot,s)\leq C s$  with $C>0$.
\end{enumerate}

Then problem \eqref{eq:pde} has a weak solution $(u_1, u_2 )$ understood in the following
sense:

\begin{itemize}
\item[(i)]  $u_{i}\geq 0$, $u_{i}\in L^{\infty}(Q_T)\cap H^{1}(0,T;(H^{1}(\Omega  ))^{'})$,

\item[(ii)]  for all $\varphi  \in L^2 (0,T; H^{1}(\Omega))$

\begin{equation}
 \int_{0}^{T}\langle \partial_t u_{i} ,\varphi \rangle
    + \int_{Q_T }  J_i(u_1,u_2) \cdot\nabla\varphi
=\int_{Q_T }u_i F(u_1+u_2)~\varphi, \label{weak.fin}
\end{equation}
where $\langle\cdot ,\cdot\rangle$ denotes the duality product of $(H^1(\Omega ))^{'}\times
H^1(\Omega )$,

\item[(iii)]  the initial conditions in \eqref{eq:pde} are satisfied in the sense

\begin{equation}
\notag \lim_{t \to 0}\| u_{i}(\cdot ,t)-u_{i0}\|_{(H^{1}(\Omega ))^{'}}=0 \quad
\text{as}\quad t\rightarrow 0.
\end{equation}
\end{itemize}
\end{theorem}
The proof of this theorem is based on the following two observations. Firstly, note that if
a weak solution of \eqref{eq:pde} does exist, then the addition of its components,
$u=u_1+u_2$ satisfies the equation

\begin{equation}
 \label{prob:suma.ec}
\partial_t u -\operatorname{div} J(u) =uF(u) \quad\text{in $Q_T$}
\end{equation}
with the flow

\begin{equation}
\label{prob:suma.flujo}
 J(u) = u\big(a\nabla u +b q \big) +
 c\nabla u,
\end{equation}
together with non-flow boundary conditions and the initial datum satisfying $u_0>0$ on
$\Omega$. Existence and uniqueness of $L^\infty (Q_T)\cap L^2(0,T;H^2(\Omega))$ positive
solutions to this uniformly parabolic problem is a well-known issue, see, e.g., \cite{LSU}.
Then,  the non-negativity of the solutions $u_i$ of problem \eqref{eq:pde} results in
$u_i\in L^\infty(Q_T)$, $i=1,2$, which is a property difficult to obtain directly from the
analysis of system \eqref{eq:pde}.

As a second observation, let us note that the usual approach to the proof of existence of
solutions to cross-diffusion systems in the most conflicting case $c_i=0$ is based on
justifying the use of $\log u_i$  as a test-function in \eqref{weak.fin} in order to obtain
estimates from the addition of the resulting identities

 \begin{eqnarray}
\label{entropy}
 \int_\Omega h(u_i(T,\cdot))+
\int_{Q_T}  \big(|\nabla u_i|^2 + \nabla u_1 \cdot \nabla u_2 \big)
 \leq C,
\end{eqnarray}
with $ h(u_i)=u_i(\log u_i-1) +1$. However, in the present case the singularity of the
diffusion matrix corresponding to \eqref{eq:pde} prevents us from obtaining the $L^2$
estimates for $\nabla u_i$ from \eqref{entropy}. To circumvent this difficulty and keep at
the same time the good properties derived for the addition of the components of a solution,
the following perturbation of the original problem is introduced:

\begin{equation}
\notag
 \label{prob:pert.ec}
\partial_t
u_{i} -\operatorname{div} J_i^{(\delta)}(u_1,u_2)=u_iF(u)\quad \text{in $Q_T$}
\end{equation}
with

\begin{equation}
\label{reg:flow}
J_i^{(\delta)}(u_1,u_2)= J_i(u_1,u_2) +\frac{\delta}{2} \Delta(u_iu),
\end{equation}
subject to the non-flow boundary conditions. Using results of \cite{ggj03} one may deduce
the existence of a sequence of non-negative functions $(u_1^{(\delta)},u_2^{(\delta)})$.
Moreover, it turns out that the sum $u_{1}^{(\delta)}+u_{2}^{(\delta)}$ is uniformly bounded
in $L^\infty(Q_T)$. This fact allows one to pass to the limit, which leads to the assertion
of Theorem \ref{th:orig}. The difficulties in identifying the limit of the sequence of
solutions to the approximated problems are delivered by the diffusive and the Lotka-Volterra
terms

\[
 \int_{Q_T}u_i^{(\delta)}\nabla{(u_1^{(\delta)}+u_2^{(\delta)})}\cdot \nabla\varphi ,
\qquad  \int_{Q_T}u_i^{(\delta)}F_i(u_1^{(\delta)},u_2^{(\delta)})\varphi  .
\]
Since the $L^\infty(Q_T)$ weak-$*$ convergence is the only convergence for the independent
components $u_i^{(\delta)}$ obtained from the approximated problems, stronger conditions on
the data of the problems for $u_1^{(\delta)}+u_2^{(\delta)}$ are required in order to pass
to the limit. To be precise, one needs the strict positivity and $H^2(\Omega)$ regularity of
the initial data. Notice, however, that if a strong convergence of $u_i^{(\delta)}$ in, for
instance, $L^1(Q_T)$ is proven, then the assumptions on $u_{01}+u_{02}$ may be weakened in
such a way that just the usual $L^2(Q_T)$ weak convergence of
$\nabla{(u_1^{(\delta)}+u_2^{(\delta)})}$ holds. In addition,  in this case some other
restrictions on the coefficients, such as the equality of the diffusive terms $a_1=a_2$, or
the restriction on the form of the Lotka-Volterra terms, can be removed. In the
one-dimensional case Bertsch et al. \cite{bertsch10} proved $BV(Q_T)$ uniform estimates for
the vanishing viscosity approximation to \eqref{eq:pde12}, which allowed one to get strong
convergence in $L^1(Q_T)$. However, these estimates depend on the $L^2(Q_T)$-norm of the
Laplacian of the sum, thus leading to similar regularity assumptions on the initial data.
Let us finally notice that, due to the discontinuities arising in the limit problem, the
uniform estimate for $u_1^{(\delta)}+u_2^{(\delta)}$ in $SBV(Q_T)$ is the strongest estimate
that can be expected.

\subsection{A constructive example for the contact-inhibition problem}
We consider a particular situation of the contact-inhibition problem in which an explicit
solution of \eqref{eq:pde}  may be computed in terms of a suitable combination of the
Barenblatt explicit solution of the porous medium equation, the Heavyside function and the
trajectory of the contact-inhibition point. To be precise, we construct a solution to the
problem

\begin{align}
& \partial_t u_{i} - (u_i (u_1+u_2)_x)_x =0&& \text{in }(-R,R)\times(0,T)=Q_T,&\label{eq:s1}\\
&u_i (u_1+u_2)_x = 0 &&\text{on } \{-R,R\}\times(0,T),& \label{eq:s2}
\end{align}
with
\begin{equation}
\label{def:us}
 u_{10}(x)=H(x-x_0)B(x,0), \quad u_{20}(x)=H(x_0-x)B(x,0).
\end{equation}
Here, $H$ is the Heavyside function and $B$ is the Barenblatt solution of the porous medium
equation corresponding to the initial datum $B(x,-t^*)=\delta_0$, i.e.

\begin{equation}
\notag \label{barenblatt}
 B(x,t)= 2 (t+t^*)^ {-1/3} \big[1-\frac{1}{12}x^ 2(t+t^*)^ {-2/3}\big]_+.
\end{equation}
For simplicity, we consider problem \eqref{eq:s1}-\eqref{def:us} for $T>0$ such that
$R(T)<R^ 2$, with $R(t)=\sqrt{12}(t+t^*)^ {1/3}$, so that $B(R,t)=0$ for all $t\in[0,T]$.
The point $x_0$ is the initial contact-inhibition point, for which we assume $|x_0|<R(0)$,
i.e. it belongs to the interior of the support of $B(\cdot,0)$, implying that the initial
mass of both populations is positive.

\begin{theorem}\label{th.easy}
 The functions

\begin{equation}
\notag
 u_1(x,t)=H(x-\eta(t))B(x,t),\quad u_2(x,t)=H(\eta(t)-x)B(x,t),
\end{equation}
with $\eta(t)=x_0(t/t^*)^{1/3}$, are a weak solution of problem \eqref{eq:s1}-\eqref{def:us}
in the following sense:

\begin{equation}
\notag \int_{-R}^ {R} \big( (u_{i} \varphi)(\cdot,T) - u_{i0} \varphi(\cdot,0) \big)
 - \int_{Q_T}  u_{i} (\varphi_t -  (u_1 +u_2 )_x \varphi_x)  =  0
\end{equation}
for all $\varphi\in H^1(Q_T)$.
\end{theorem}


Let $H_\epsilon$ the regularization of the Heavyside function taking the values
$\left\{1,\frac{1}{2}(1-x/\epsilon),0\right\}$ in the intervals $(-R,-\epsilon)$,
 $(-\epsilon,\epsilon)$ and $(\epsilon,R)$, respectively, for $\epsilon>0$ small. The proof of the above theorem is based on the approximation result given in the next lemma.

\begin{lemma}
 Let $u_i^ \epsilon:[0,T]\times[-R,R]$, $i=1,2$, be given by

\begin{equation}
\label{def:us2} u_1^ \epsilon(x,t)=H_\epsilon(x-\eta(t))B(x,t), \quad u_2^
\epsilon(x,t)=H_\epsilon(\eta(t)-x)B(x,t)
\end{equation}
with $\eta(t)=x_0(t/t^*)^{1/3}$. Then

\begin{equation}
\notag \left|\int_{-R}^ {R} \big( (u_{i\epsilon} \varphi)(\cdot,T) - (u_{i\epsilon}
\varphi)(\cdot,0)\big) - \int_{Q_T}  u_{i}^ \epsilon (\varphi_t -  (u_1^ \epsilon+u_2^
\epsilon )_x \varphi_x)  \right|\leq C\epsilon
\end{equation}
for all $\varphi\in H^ 1((t^*,T)\times(-R,R))$.

\end{lemma}

\begin{proof} Observe that $u_{i}^{\epsilon}$ are continuous and bounded in $\Omega\times(t^ *,T)$, and
satisfy $u_1^ \epsilon+u_2^ \epsilon =B$. Therefore, $u_1^ \epsilon+u_2^ \epsilon \in
L^2(t^*,T;H^1(-R,R))$ uniformly in $\epsilon$. Let $\varphi\in H^ 1(Q_T)$.  Using
$\varphi_\epsilon (x,t)=\varphi(x,t) H_\epsilon(x-\eta(t)) $ as the test-function in the
weak formulation of the problem satisfied by the Barenblatt solution in $Q_T$ we obtain

\begin{equation}
\notag \label{weak.B}
 \int_{-R}^ {R} \big( (H_\epsilon B \varphi)(\cdot,T) - (H_\epsilon B \varphi)(\cdot,t^*)\big) -
 \int_{Q_T}  H_\epsilon B (\varphi_t  - B_x \varphi_x )   =I^ 1_\epsilon,
\end{equation}
with

\begin{eqnarray*}
I^ 1_\epsilon
  = -\int_{Q_T} \varphi(x,t) B(x,t) H'_\epsilon(x-\eta(t)) \big(\eta'(t) + B_x(x,t)\big) dx dt.
\end{eqnarray*}
Since $|x_0|<R(0)$, we have $\eta(t)<R-\epsilon$, for $\epsilon$ small enough and
$t\in(0,T)$, and then using the explicit expression of $B_x$ and $\eta'(t)$ we deduce

\begin{equation}
\notag \label{eq.I} I_\epsilon^1=\frac{1}{6\epsilon}\int_{0}^T\int_{-\epsilon}^\epsilon y
\varphi(y+\eta(t),t)B(y+\eta(t),t)dydt.
\end{equation}
Since $\varphi$ and $B$ are uniformly bounded in $L^\infty$, we obtain

\begin{equation}
\label{est.I}
 |I^ 1_\epsilon|\leq C \epsilon ,
\end{equation}
with $C>0$ independent of $\epsilon$. The computation using
$\varphi(x,t)H_\epsilon(\eta(t)-x)$ as test function gives similar results for some
$I_\epsilon^ 2$ satisfying the same estimate \eqref{est.I} than $I_\epsilon^ 1$. Observing
that functions \eqref{def:us2} satisfy $u_1^\epsilon+u_2^\epsilon=B$, we finish the proof.
\end{proof}

\begin{proof}[Proof of Theorem \ref{th.easy}] Since $u_i^ \epsilon$ are uniformly bounded
in $L^ \infty (Q_T)$ we may perform the limit $\epsilon\to 0$ to deduce, on one hand, the
existence of $u_i\in L^ \infty(Q_T)$ such that

\begin{equation}
\notag \int_{-R}^ {R} \big( (u_{i} \varphi)(\cdot,T) - u_{0i} \varphi(\cdot,0)\big)
 - \int_{Q_T}  u_{i} (\varphi_t -  (u_1 +u_2 )_x \varphi_x)  =  0.
\end{equation}
On the other hand, taking the limit of expressions \eqref{def:us2} we get

\begin{equation}
\notag
 u_1(x,t)=H(x-\eta(t))B(x,t),\quad u_2(x,t)=H(\eta(t)-x)B(x,t).
\end{equation}
\end{proof}

\begin{remark}
The problem solved by $\eta$ is related to $B$ by the ODE problem

\begin{equation}
\notag
\label{eq.eta} \left\{
 \begin{array}{ll}
  \eta'(t)=-B_x(t,\eta(t)) & \text{for }t\in (0,T), \\
  \eta(0)=x_0,
 \end{array}
\right.
 \end{equation}
which ensures the mass conservation for each component. Indeed, defining

\begin{equation}
\notag
 M_i(t)=\int_{-R}^R u_i(x,t)dx = \int_{-R}^{ \eta(t)} B(x,t)dx,
\end{equation}
we find, using the equation satisfied by $B$ and its boundary conditions

\begin{equation}
\notag
\begin{split}
M'_i(t) & = \int_{-R}^{ \eta(t)} B_t(x,t)dx
+\eta'(t)B(\eta(t),t)
\\
& =B(\eta(t),t)B_x(\eta(t),t)+\eta'(t)B(\eta(t),t)=0.
\end{split}
\end{equation}
 \end{remark}
\begin{remark}
It is not difficult to extend the above construction to other one-dimensional problems. For
instance, for problem \eqref{eq:pde} we may consider the solution $u$ of
\eqref{prob:suma.ec}-\eqref{prob:suma.flujo} and the corresponding approximations of the
type \eqref{def:us2}. Then, to handle the integrals $I_\epsilon^i$, we first observe that
for $\epsilon\to0$ we get
\begin{eqnarray*}
 I^ 1_\epsilon \to -\int_0^T \varphi(\eta(t),t) B(\eta(t),t) \big(\eta'(t) + G(\eta(t),t)\big) dt,
\end{eqnarray*}
with $G = a u_x +b q + c (\log(u))_x.$ Therefore, if the ODE problem
\begin{equation}
\label{eq.eta2}
\left\{
 \begin{array}{ll}
  \eta'(t)=-G(t,\eta(t)) & \text{for }t\in (0,T), \\
  \eta(0)=x_0,
 \end{array}
\right.
 \end{equation}
is solvable, a solution for problem \eqref{eq:pde} may be constructed. Typical conditions on
$G$ for \eqref{eq.eta2} to be solvable are given in terms of Sobolev or BV regularity in
space for $G$ and $L^1(0,T;L^\infty (-R,R))$ regularity for the divergence of $G$, $G_x$ in
the one-dimensional case, see \cite{diperna89,ambrosio04} for further details.
\end{remark}

\subsection{Numerical experiments}

The discretization of \eqref{eq:pde} with the regularizing term given in \eqref{reg:flow}
follows the standard Finite Element methodology. To construct a solution we apply the
semi-implicit Euler scheme in time and a $\mathbb{P}_1$ continuous finite element
approximation in space and then study the behavior of solutions as $\delta\to 0$, see
\cite{galiano12b} for the details.

Let   $\tau>0$ be the time step of the discretization. For $t=t_0=0$, set $u_{\epsilon
i}^0=u_i^0$. Then, for $n\geq 1$ the problem is to find $u_{\epsilon
i}^{n}:(0,T)\times\Omega\to\mathbb{R}$ such that for,  $i=1,2$,
\begin{equation}\label{eq:pde_discr.s4}
\begin{array}{l}
\frac{1}{\tau}\big( u^n_{\epsilon i}-u^{n-1}_{\epsilon i} , \chi )^h + \big(
J^{(\delta)}_i(\Lambda _{\epsilon } (u^n_{\epsilon 1}),\Lambda _{\epsilon } (u^n_{\epsilon
2}),\nabla u^n_{\epsilon 1},\nabla u^n_{\epsilon 2}  ) ,\nabla\chi \big)^h =\\ [2ex]
\hspace*{1cm} = \big(\alpha_{i} u^n_{\epsilon i} - \lambda _{\epsilon } (u^n_{\epsilon i}) (
\beta_{i1} \lambda _{\epsilon } (u^{n-1}_{\epsilon 1}) + \beta_{i2} \lambda
_{\epsilon } (u^{n-1}_{\epsilon 2}) ) , \chi \big)^h , 
\end{array}
\end{equation}
for every $ \chi\in S^h $, the finite element space of piecewise $\mathbb{P}_1$-elements.
Here, $(\cdot,\cdot)^h$ stands for a discrete semi-inner product on
$\mathcal{C}(\overline{\Omega} )$. The parameter $\epsilon>0$ makes reference to the
regularization introduced by functions $\lambda_\epsilon$ and $\Lambda_\epsilon$, which
converge to the identity as $\epsilon\to0$.

Since \eqref{eq:pde_discr.s4} is a nonlinear algebraic problem, we use a fixed point
argument to approximate its solution,  $(u_{\epsilon 1}^n,u_{\epsilon 2}^n)$, at each time
slice $t=t_n$, from the previous approximation $u_{\epsilon i}^{n-1}$.  Let $u_{\epsilon
i}^{n,0}=u_{\epsilon i}^{n-1}$. Then, for $k\geq 1$ the problem is to find $u_{\epsilon
i}^{n,k}$ such that for $i=1,2$, and for all $\chi \in S^h$
\begin{equation*}
\begin{array}{l}
 \frac{1}{\tau}\big( u^{n,k}_{\epsilon i}-u^{n-1}_{\epsilon i} , \chi )^h
+ \big( J^{(\delta)}_i(\Lambda _{\epsilon } (u^{n,k-1}_{\epsilon 1}),\Lambda _{\epsilon }
(u^{n,k-1}_{\epsilon 2}),\nabla u^{n,k}_{\epsilon 1},\nabla u^{n,k}_{\epsilon 2}  )
,\nabla\chi \big)^h =\\ [2ex] \hspace*{1cm} = \big(\alpha_{i} u^{n,k}_{\epsilon i} - \lambda
_{\epsilon } (u^{n,k-1}_{\epsilon i}) ( \beta_{i1} \lambda _{\epsilon } (u^{n-1}_{\epsilon
1}) + \beta_{i2} \lambda
_{\epsilon } (u^{n-1}_{\epsilon 2}) ) , \chi \big)^h . 
\end{array}
\end{equation*}
We use the stopping criteria $\max _{i=1,2}
\|u_{\epsilon,i}^{n,k}-u_{\epsilon,i}^{n,k-1}\|_\infty <\text{tol}$, for empirically chosen
values of $\text{tol}$, and set $u_i^n=u_i^{n,k}$.

\begin{figure}[t]
\centering
 \subfigure
 {\includegraphics[width=4.cm,height=3.cm]{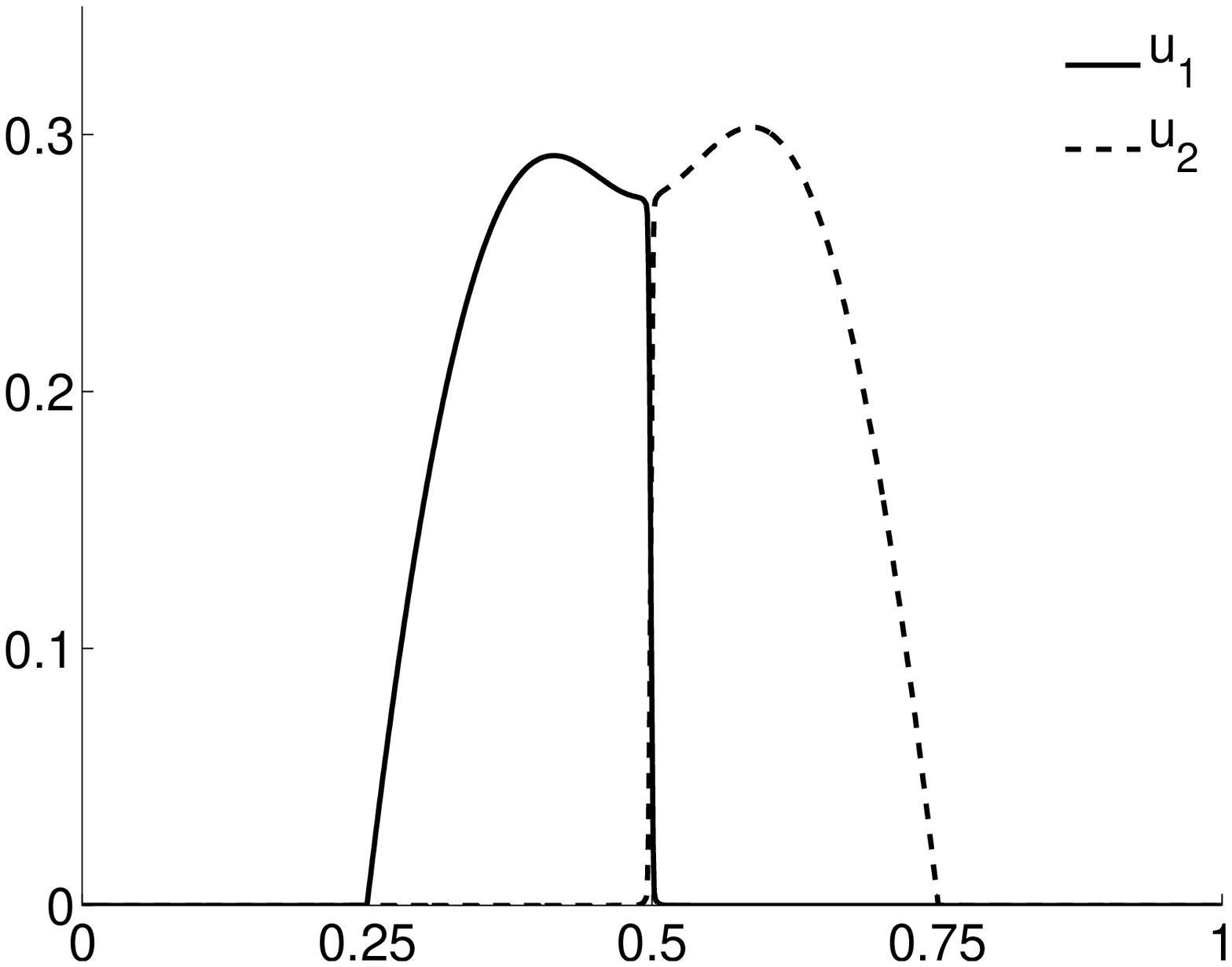}}
 {\includegraphics[width=4.cm,height=3.cm]{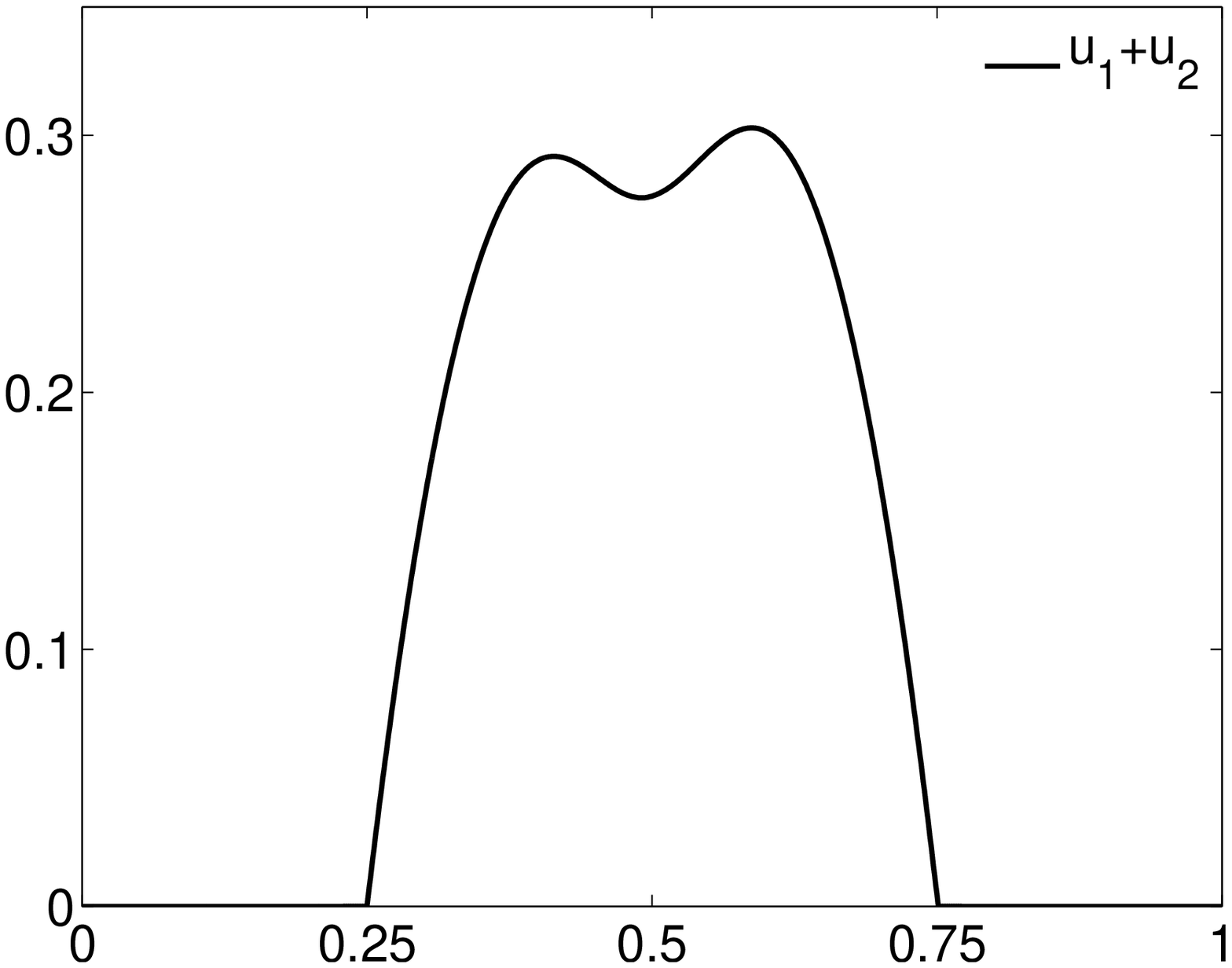}}
 {\includegraphics[width=4.cm,height=3.cm]{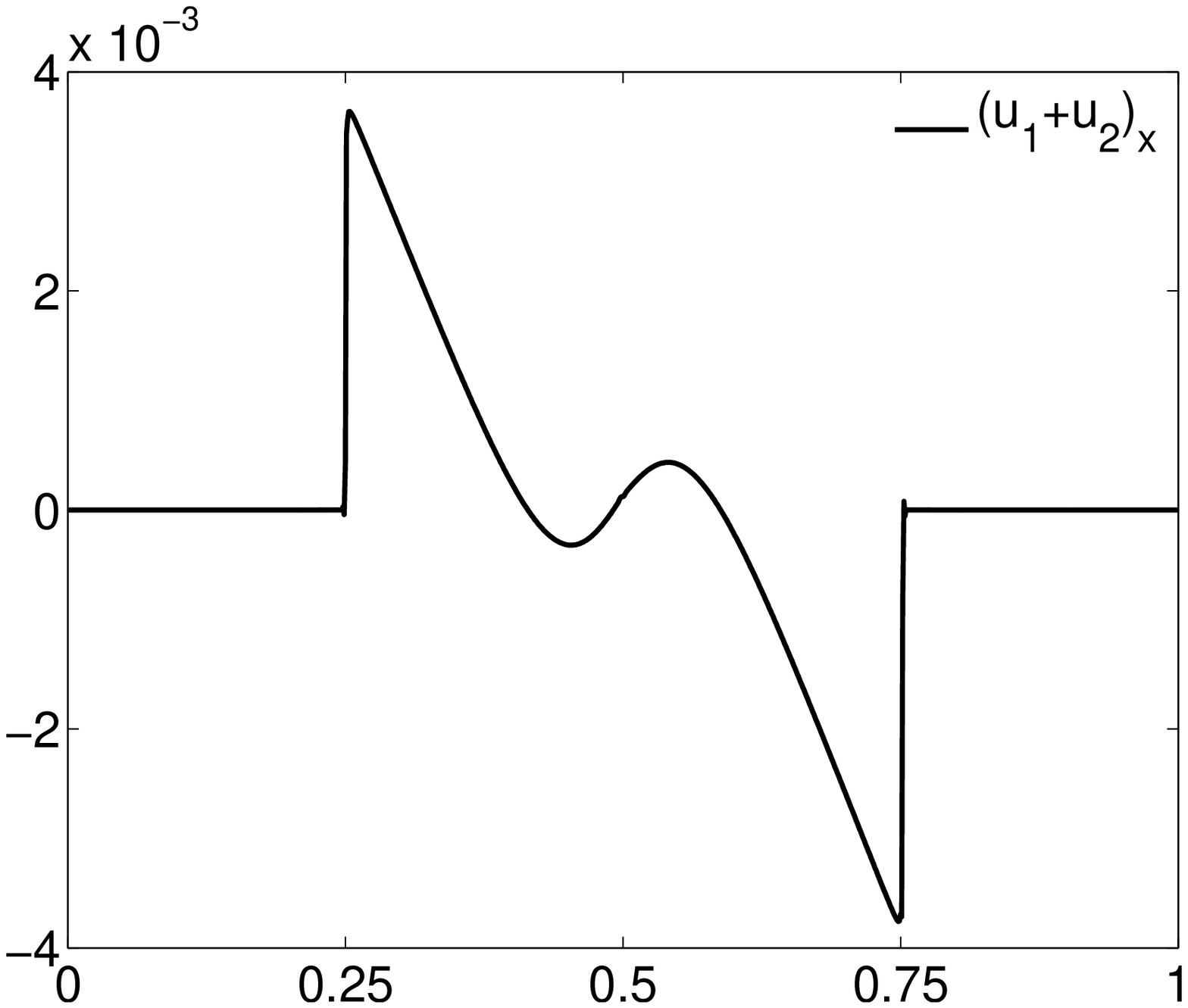}}\\
  \subfigure
 {\includegraphics[width=4cm,height=3cm]{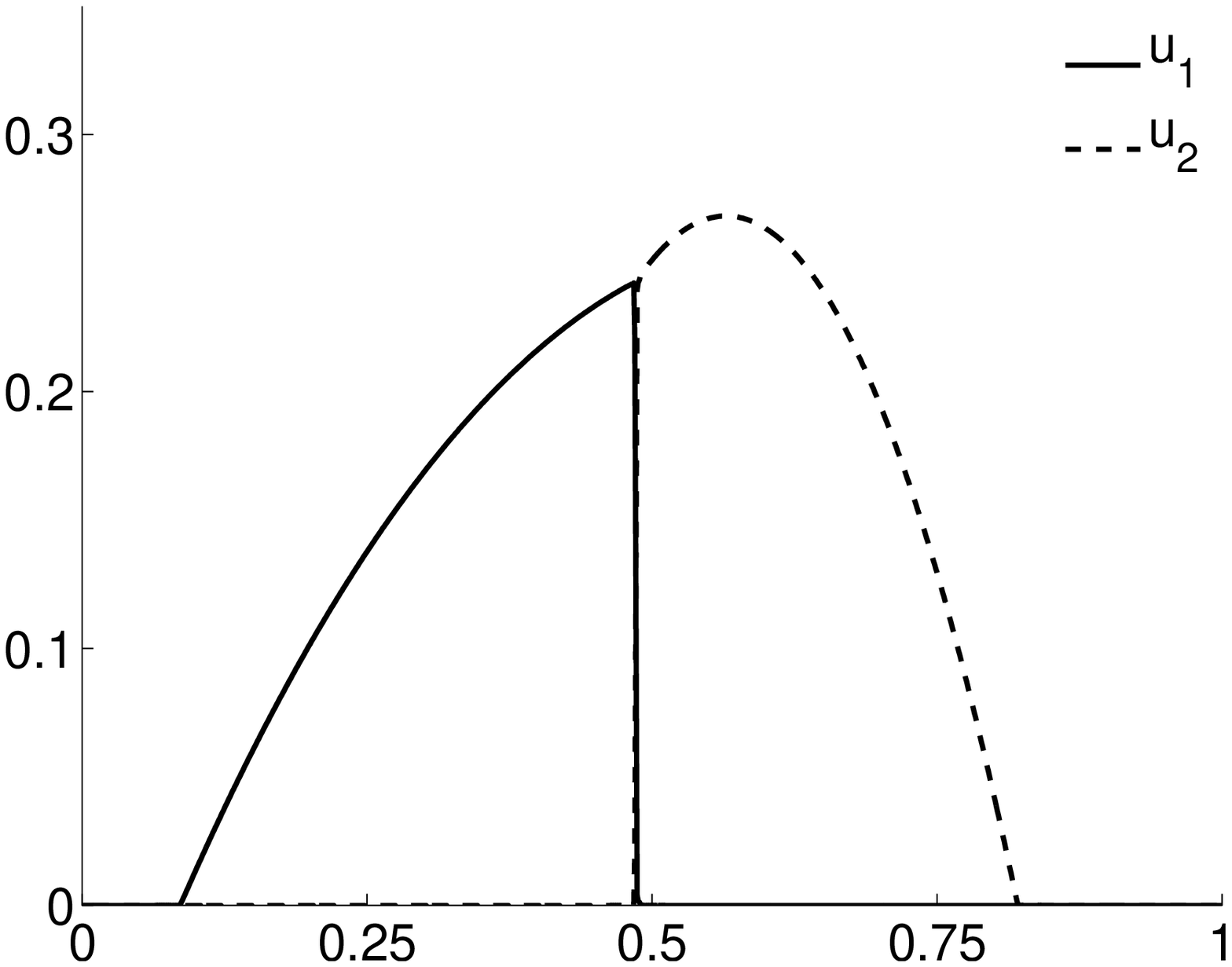}}
 {\includegraphics[width=4cm,height=3cm]{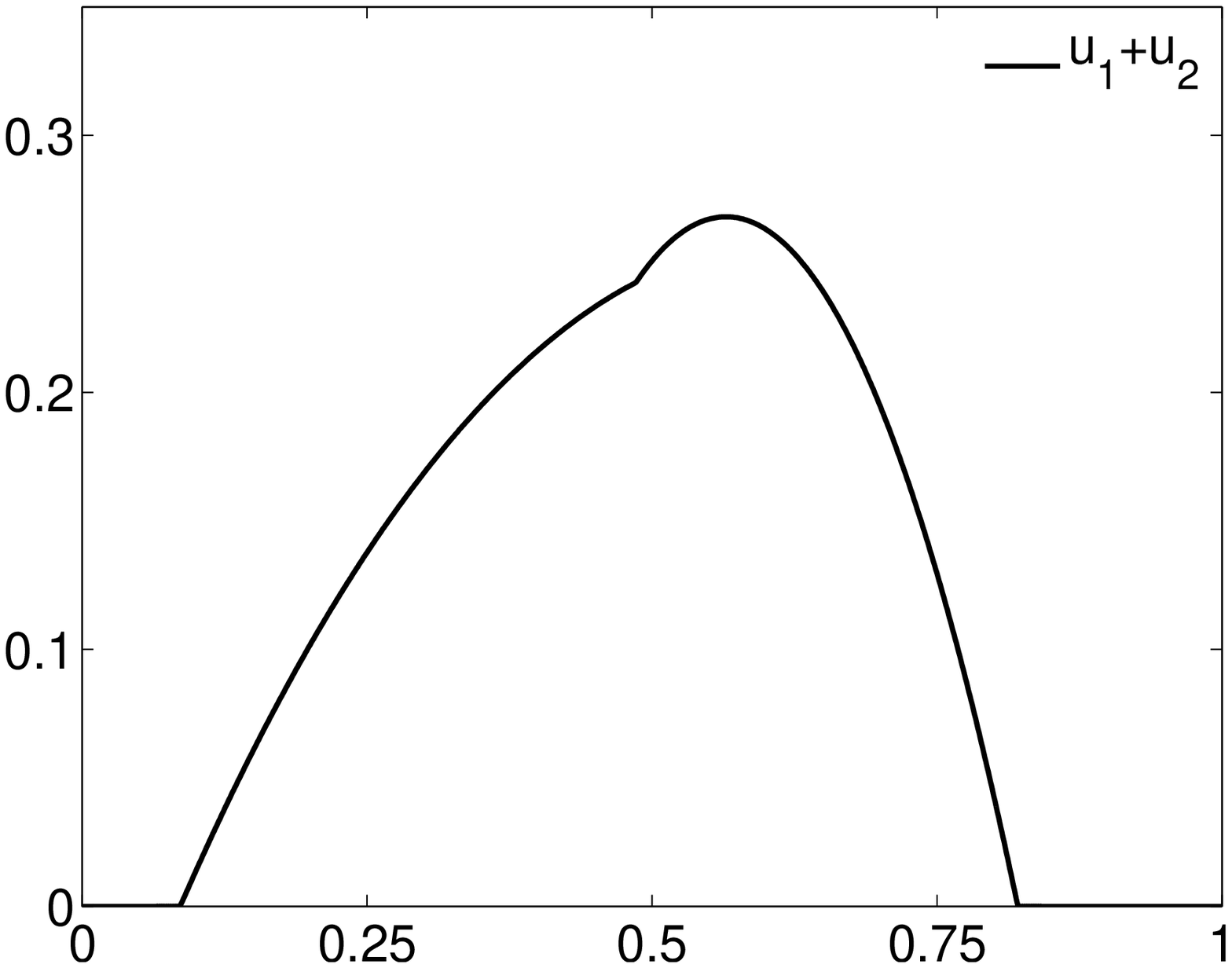}}
 {\includegraphics[width=4cm,height=3cm]{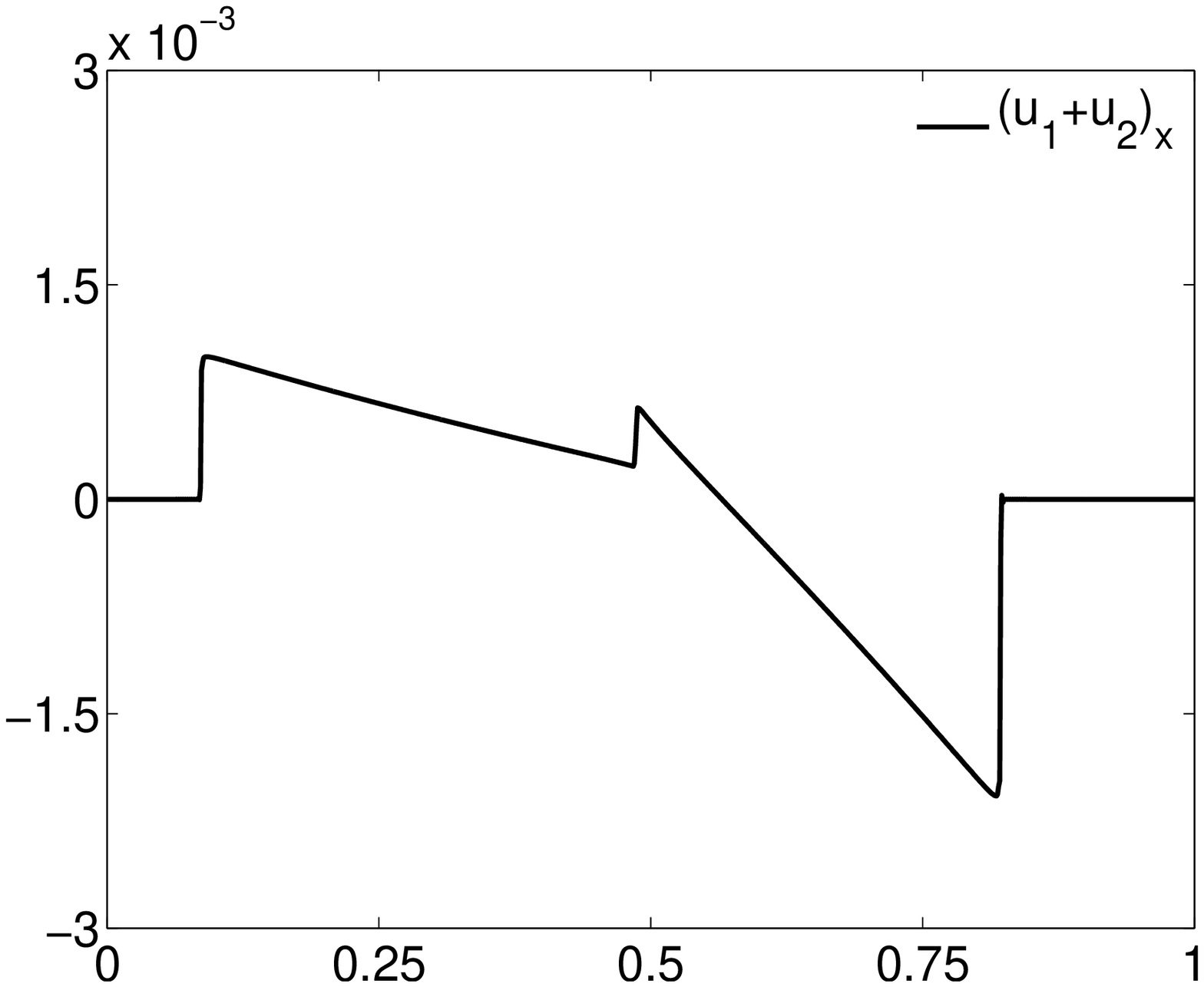}}
\caption{{\small A transient state of solutions of Experiments 1 (first row) and 2 (second
row). Left panel: solutions $(u_1^{(\delta)},u_2^{(\delta)})$. Center panel: The sum
$u^{(\delta)}=u_1^{(\delta)}+u_2^{(\delta)}$. Right panel: the space derivative of the sum,
$u_x^{(\delta)}$. }} \label{exp01_fig}
\end{figure}

In the following experiments we take a uniform partition of $\Omega=(0,1)$ in $10^3$
subintervals and the time step $\tau = 10^{-5}$.  The drift and the linear diffusion
coefficients are $b_i=c_i=0$, and the Lotka-Volterra terms, i.e. the right-hand side of
\eqref{eq:pde} have the form $f_i(u_1,u_2)=u_i(\alpha_i-\beta_{i1}u_1-\beta_{i2}u_2)$ with $
\alpha_1=1,~\beta_{11}=1,~\beta_{12}=0.5,~\alpha_2=5,~\beta_{21}=1$,and $\beta_{22}=2$. For
the initial data we take $u_{i0}= \exp((x-x_i)^2/0.001)$, $f_i=0$ for $i=1,2$ with $x_1=0.4$
and $x_2=0.6$. Although the initial data  do not satisfy the condition $u_{10}+u_{20}>0$ in
$\Omega$, this does not seem to affect the convergence or stability of the algorithm for the
cases under study. Finally, the tolerance parameter for the fixed point algorithm is set to
$\text{tol}=10^{-4}$, and the perturbation parameter to $\delta=10^{-3}$.

We run two experiments according to different nonlinear diffusion matrices. In the first
experiment, we set the same diffusion coefficient $a=1$ for both equations, which is the
situation studied in Theorems \ref{th:bertsch} and \ref{th.gs}. In the second experiment we
take different diffusivities, $a_1=1$ and $a_2=3$, in the equations for $u_1$ and $u_2$ (see
\eqref{eq:pde}). The aim of these experiments is to confirm numerically that, unlike the
case of equal diffusivities,  in our case the gradient of the sum $u_1+u_2$ may develop
discontinuity. This property can be checked on Figure~\ref{exp01_fig}. In the first row we
show the results for a transient state of the equal-diffusivities case. Although the
independent components of the solution, $u_1$ and $u_2$ exhibit a discontinuity at the
contact point, $x=0.5$, the sum $u_1+u_2$ is continuous and, as it can be seen in the right
panel of the first row, the derivative seems to be continuous as well. In the second row of
Figure~\ref{exp01_fig} we show the results corresponding to the different diffusivities
case. The behavior is clearly different. Although the continuity of $u_1+u_2$ still holds, a
discontinuity of $(u_1+u_2)_x$ at the contact point may be observed.

\end{document}